\documentclass[a4paper,10pt,reqno]{amsart}

\usepackage[utf8]{inputenc}
\usepackage{microtype}
\usepackage{varioref}
\usepackage{fancyhdr}
\usepackage[english]{babel}
\usepackage{amsthm}
\usepackage{amsmath}
\usepackage[T1]{fontenc}
\usepackage{amssymb}
\usepackage{mathtools}
\usepackage{enumitem}

\theoremstyle{remark}

\newtheorem{rmk}{Remark}[section]
\theoremstyle{plain}
\newtheorem{ex}{Example}[section]
\newtheorem{lem}{Lemma}[section]
\newtheorem{thm}{Theorem}[section]
\newtheorem{prop}{Proposition}[section]

\theoremstyle{definition}
\newtheorem{defin}{Definition}[section]
\newtheorem{hyp}{Hypothesis}[section]

\newcommand{\Kk}{\mathcal K}

\newcommand{\N}{\mathbb N}
\newcommand{\R}{\mathbb R}
\newcommand{\E}{\mathbb E}
\newcommand{\Hh}{\mathbb H}
\newcommand{\K}{\mathbb K}

\newcommand{\Pp}{\mathbb P}
\newcommand{\F}{\mathcal F}
\newcommand{\U}{\mathcal U}
\newcommand{\Uu}{\mathbb U}
\DeclarePairedDelimiter{\abs}{\lvert}{\rvert}
\DeclarePairedDelimiter{\norma}{\lVert}{\rVert}

\DeclarePairedDelimiter{\prodscal}{\langle}{\rangle}

\allowdisplaybreaks

\numberwithin{equation}{section}

\title[A class of Semi-Linear Backward Parabolic Cauchy Problems]{A Semi-Linear
Backward Parabolic cauchy Problem with Unbounded Coefficients of
Hamilton-Jacobi-Bellman Type and Applications to optimal control}
\author{Davide Addona}
\address{Dipartimento di Matematica, Universit\`a degli Studi di Milano Bicocca,
Via Cozzi, 53, I-20155 Milano (Italy)}
\email{d.addona@campus.unimib.it}

\begin{document}



\begin{abstract}
We obtain weighted uniform estimates for the gradient of the solutions to a class of linear parabolic Cauchy problems with unbounded coefficients. Such estimates are then used to prove existence and uniqueness of the mild solution to a semi-linear backward parabolic Cauchy problem, where the differential equation is the Hamilton-Jacobi-Bellman equation of a suitable optimal control problem. Via backward stochastic differential equations, we show that the mild solution is indeed the Value Function of the controlled equation and that the feedback law is verified.
\end{abstract}

\maketitle

\section{Introduction}

The aim of this paper is to study of the backward parabolic Cauchy problem
(BPDE) of HJB type
\begin{equation}
\left\{
\begin{array}{ll}
D_tv(t,x)+Av(t,x)=\psi(x,G(x)\nabla
v(t,x)), \quad t\in[0,T), & x\in\R^N, \\
\\
v(T,x)=\varphi(x), & x\in\R^N,
\end{array}
\right.
\label{int_eq:BPDE}
\tag{BPDE}
\end{equation}
by analytic methods, and show some of its applications to stochastic optimal
control problems.

Here, $A$ is the uniformly elliptic differential operator defined on smooth
functions $f$ by
\[
Af(x)=\frac12Tr[G(x)G(x)D^2_xf(x)]+ \prodscal{B(x),\nabla f(x)},
\]
where $G:\R^N\longrightarrow\R^N\times\R^N$, $B:\R^N\longrightarrow\R^N$, $\psi$ is a continuous function which satisfies some additional conditions
and $\varphi$ is a bounded and continuous function in
$\R^N$ (for short $f\in C_b(\R^N)$). The particular form of the nonlinear term
arises naturally in the theory of stochastic control (see \cite{fuhr-tess2}, \cite{pardoux-peng}), but it has also an
analytic interest.

In recent years much attention has been paid to the uniformly elliptic operator
$A$, with unbounded coefficients in $\R^N$, since they naturally appear in the
theory of Markov processes. If $f\in C_b(\R^N)$, under suitable hypothesis the
Cauchy problem
\begin{equation}
\left\{
\begin{array}{ll}
D_tu(t,x)=Au(t,x), \quad t>0, & x\in\R^N, \\
\\
u(0,x)=f(x), & x\in\R^N,
\end{array}
\right.
\label{int_eq:cauchy_problem_omo}
\end{equation}
has a unique bounded solution. Moreover, there exists a semigroup
$\{S(t)\}_{t\geq0}$ of linear operators in $C_b(\R^N)$ such that
$u(t,x)=S(t)f(x)$. In general, this semigroup is neither strongly continuous nor
analytic (see \cite{meta-prio}). The problem of estimating the gradient of the solution
$u$ has already been studied in literature by both analytic (\cite{bert-forn}, \cite{bert-forn-lor}, \cite{bert-forn-lor2}, \cite{bert-lor}, \cite{daprato},
\cite{lunardi}) and probabilistic methods
(\cite{cerrai}, \cite{delarue}).

In this paper, under suitable assumptions on the coefficients of the operator $A$, we
prove the existence and uniqueness of a mild solution to problem
\eqref{int_eq:BPDE}. This is not a straightforward task since both $G$ and $B$
may be unbounded. More precisely, let $\{S(t)\}_{t\geq0}$ be the semigroup associated to
the Cauchy problem \eqref{int_eq:cauchy_problem_omo}, and $F$ be the functional
defined by
\[
F(t,u)(x)=\psi(x,G(x)\nabla u(t,x)), \quad t\in[0,T), \ x\in\R^N,
\]
for suitable functions $u$. We show that the functional
\[
(\Gamma v)(t,x):=S(T-t)\varphi(x)-\int_t^TS(r-t)F(r,v)(x)dr,
\]
admits a unique fixed point if $v\in C_b([0,T]\times\R^N)\bigcap C^{0,1}([0,T)\times\R^N)$ which
satisfies the following growth condition:
\[
\norma{G\nabla v(t,\cdot)}_\infty\leq\frac C{(T-t)^{1/2}}\norma\varphi_\infty,
\quad t\in[0,T),
\]
for some $C>0$.

The novelty and issue of this estimate is the presence of an unbounded coefficient which multiplies the gradient; indeed no unbounded function is present in the classical gradient estimates, obtained both by analytic and probabilistic methods. As one expects, the presence of this term yields to additional growth conditions on the coefficients of the operator $A$, since we require not only that the gradient of the solution is bounded, but that its decreasing could balance the growth of $G$.


%
%
%

As it is well known Equation \eqref{int_eq:BPDE} is the Hamilton Jacobi Bellman (HJB) equation
corresponding to an optimal stochastic control problem. Namely If $\varphi\in BUC(\R^N)$ the regularity of the mild
solution $v$ allows us to show that it is
exactly the Value Function associated to the control problem given by the state equation
\begin{equation}
\left\{
\begin{array}{ll}
D_\tau X^u_\tau=B(X^u_\tau)d\tau+G(X^u_\tau)r(X^u_\tau,u_\tau)d\tau+G(X^u_\tau)dW_\tau, &
\tau\in[t,T], \\
\\
X_t^u=x\in\R^N,
\end{array}
\right.
\label{int:controlled_eq}
\end{equation}
and the cost functional is
\begin{equation}
\E\int_0^Tl(X^u_t,u_t)dt+\E\varphi(X^u_T),
\label{int_eq:cost_functional}
\end{equation}
where $l,\varphi$ are measurable functions.

The existence of $\nabla_xv$ and the estimate on its growth allow us to
identify the optimal feedback law for the control problems.

The key tool to link (BPDE) and the controlled equation are the backward
stochastic differential equations. This connection was proved in the paper
\cite{pardoux-peng} for the finite dimensional case and for classical
solutions of the parabolic Cauchy problem
\[
\left\{
\begin{array}{lll}
\displaystyle \frac{\partial u}{\partial t}(t,x)+\mathcal
Lu(t,x)+f(t,x,u(t,x),(\nabla u\ \sigma)(t,x))=0, & t\in[0,T], & x\in\R^N, \\
\\
u(T,x)=g(x), & & x\in\R^N,
\end{array}
\right.
\]
where
\[
\mathcal L=\frac12\sum_{i,j=1}^N(\sigma\sigma^*)_{ij}(t,x)\frac{\partial
^2}{\partial x_ix_j}+\sum_{i=1}^Nb_i(t,x)\frac{\partial}{\partial x_i},
\]
$\sigma$ is a $(N\times d)-$matrix valued function and $b_i$ are scalar functions, for $i=1,\ldots,N$.
For the infinite dimensional case, we refer to \cite{fuhr-tess2} where the authors prove that the solution to the backward stochastic differential equation is also the
unique mild solution of a suitable partial differential equation.

It is well known that the forward backward stochastic differential equation we deal with, which is
\begin{equation}
\left\{
\begin{array}{ll}
dY_\tau=\psi(X_\tau,Z_\tau)d\tau+Z_\tau dW_\tau, & \tau\in[t,T], \\
\\
dX_\tau=B(X_\tau) d\tau+G(X_\tau) dW_\tau, & \tau\in[t,T], \\
\\
Y_T=\varphi(X_T), \\
\\
X_t=x, & x\in\R^N,
\end{array}
\right.
\label{int_eq:FBSDE}
\tag{FBSDE}
\end{equation}
has a solution $(X,Y,Z)$ with $X,Y,Z$ belonging to some suitable spaces,
and under opportune regularity and growth assumptions on $\psi,B,G,\varphi$ the
processes $Y$ and $Z$ are indeed $v$ and $G\nabla v$, respectively (see
\cite{pardoux-peng}). These assumptions are quite strong, since they require
differentiability for $G,B,f$ and $g$. Our analytic results allow us to obtain
these identifications relaxing the hypotheses on the terms of the Cauchy
problem, and so to study the control problem in a more general setting. We also
notice that the needed regularity could be obtained by Bismut-Elworthy formula
but such an approach was exploited in letterature, at our best knowledge, only
in the case of a bounded diffusion, see \cite{cerrai} and \cite{fuhr-tess}.

The paper is organized as follows. In Section \ref{sec:BPDE} we prove the
existence and uniqueness of a mild solution to (BPDE), and study some of its
regularity properties.

In the first subsection, we show that the estimate
\[
\norma{G\nabla S(t)\varphi}_\infty\leq \frac
C{t^{1/2}}\norma\varphi_\infty,
\quad \forall t\in(0,T],
\]
holds for any $\varphi\in C_b(\R^N)$, any $T>0$ and some positive constant $C=C(T)$.

In the second subsection, we prove some regularity results for the mild solution $v$ of \eqref{int_eq:BPDE}. Moreover, a classical fixed point argument shows the existence and uniqueness of a local solution to the Cauchy problem \eqref{int_eq:BPDE}, solution which can be extended to the line $(-\infty,T]$.

The second part of the paper is devoted to the study of \eqref{int_eq:FBSDE}
which, as we stressed above, is the key tool to prove that $v$ is indeed the Value Function associated
to problem \eqref{int:controlled_eq}.

Finally, in Section $4$ we introduce the stochastic controlled equation.
The regularity of $v$ and the solvability of \eqref{int_eq:FBSDE} enable us to prove that $v$ is the
value function and that, under suitable assumptions, the feedback law is
verified.

\subsubsection*{{\bf{Notation}}}
Throughout the paper we denote by $B(R)$ the open ball in $\R^N$ with center at
$x=0$ and radius $R$, and by $\overline{B(R)}$ its closure.

\section{The Semi-Linear PDE}
\label{sec:BPDE}

Let us consider the backward Cauchy problem
\begin{equation}
\left\{
\begin{array}{ll}
D_tu(t,x)+Au(t,x)=\psi(x,G(x)\nabla
u(t,x)), \quad t\in[0,T), & x\in\R^N, \\
\\
u(T,x)=\varphi(x), & x\in\R^N,
\end{array}
\right.
\label{eq:eq_Cauchy_NL}
\end{equation}
where $A$ is the second order elliptic operator, defined on smooth functions
$f:\R^N\longrightarrow\R$ by
\begin{equation}
Af(x)=\frac12Tr[Q(x)D^2_xf(x)]+\prodscal{B(x),\nabla
f(x)},
\label{eq:op_diff_A}
\end{equation}
$Q(x)$ is a positive defined matrix for any $x\in\R^N$, $G=\sqrt Q$, $\varphi\in C_b(\R^N)$, and $\psi$ is a continuous function, which
satisfies the following conditions:
\begin{hyp}
\begin{description}
\item [(i)] $\psi(\cdot,0)$ is bounded in $\R^N$.

\item[(ii)] For some constant $L_\psi>0$ and any $x,x_1,x_2,y_1,y_2\in\R^N$ we have
\begin{equation}
\begin{split}
\abs{\psi(x_1,x_2)-\psi(y_1,y_2)} & \leq
L_\psi\abs{x_2-y_2}+
L_\psi\abs{x_1-y_1}\left(1+\abs{x_2}+\abs{y_2}\right), \\
\abs{\psi(x,0)} & \leq L_\psi.
\end{split}
\label{eq:g_lipschitz}
\end{equation}
\end{description}
\label{hyp:g_lipschitz}
\end{hyp}

We introduce some definitions, to enlighten the computations in the next pages:
for any $i=1,\ldots,N$ and any $x\in\R^N$, we set
\begin{align}
f_i(x) & = \abs*{\sum_{j=1}^NQ_{ij}(x)\left(D_jG(x)\right)G^{-1}(x)}, \label{eq1:appr_fi} \\
h(x)^\gamma & = \sum_{j,k,l,m=1}^N\abs*{G_{jk}(x)D_kG_{lm}(x)}^\gamma, \quad\forall \gamma>0, \label{eq1:appr_h} \\
l^i_R(x) & =\frac{1}{1+R^2}\abs*{\sum_{j=1}^NQ_{ij}(x)x_j}, \quad \forall R\geq1 \label{eq1:appr_ljR}.
\end{align}

Now we can state the growth hypotheses on the coefficients $Q_{ij}(x), B_i(x)$;
\begin{hyp}
\begin{description}
\item [(i)] $B_i\in C^1(\R^N)$ and $Q_{ij}\in
C^2(\R^N)$, for any $i,j=1,\ldots,N$.

\item [(ii)]
$Q(x)$ is a uniformly positive-definite matrix, i.e. there exist a positive
function $\nu$ and a constant $\nu_0>0$, such that $\nu(x)\geq\nu_0>0$, for any
$x\in\R^N$, and
\[
\prodscal{Q(x)\xi,\xi}\geq\nu(x)\abs\xi^2, \quad \textrm{for any }\xi,x\in\R^N,
\]
and $B$ is uniformly dissipative, which means that
\[
\prodscal{B(x)\xi,\xi}\leq0, \quad \forall \xi\in\R^N, \quad ,\forall x\in\R^N;
\]

\item [(iii)]
There exist a positive function $b$ and a positive constant $b_0$ such that
\begin{equation}
-M(x)\geq b(x)\geq b_0>0,
\label{hyp0:DB}
\end{equation}
where $M$ is the matrix-valued function defined by
\begin{equation}
M:=G\left(DB\right)G^{-1}-\sum_{i,j=1}^NQ_{ij}\left(D_{ij}G\right)G^{-1}
-\sum_{j=1}^NB_j\left(D_jG\right)G^ {-1};
\end{equation}

\item [(iv)]
growth conditions on $Q$ and $B$: there exist positive constants $K_j$, $j=1,\ldots,6$, $C_n$
($n\in\N$), and $\delta\in[0,3/2]$, $\alpha,\beta\in[0,2]$ such that

\begin{align}
&\max_{j=1,\ldots,N}\abs*{\sum_{i=1}^NQ_{ij}(x)x_i}^\delta\leq K_1(1+\abs x^2)^\delta\nu(x)
\label{hyp1:growth_Q}, \ \forall x\in\R^N,\\
\notag\\
& K_2\sum_{j=1}^N\abs*{Q_{ij}(x)x_i}l_R(x)^{3-2\delta}
+4\abs*{x_i}f_i(x)+x_iB_i(x)\leq K_3(1+\abs x^2), \label{hyp3:growthB} \\
& \qquad \forall i=1,\ldots,N, \quad \forall \abs x\leq R, \quad R\geq1,
\notag \\
& K_4\left[\left(\frac{\langle Q(x)x,x\rangle}{1+\abs x^4}\right)^2
+\left(\frac{Tr(Q(x))}{1+\abs x^2}\right)^2\right]-b(x)\leq K_5,  \quad \forall
x\in\R^N,
\label{hyp4:growthDB_1}\\
\notag\\
& \left\{
\begin{array}{ll}
\displaystyle n\left(\sum_{i=1}^Nf_i(x)^\alpha+h(x)^\beta\right)-b(x)\leq C_n, & \forall
x\in\R^N,\ \forall n\in\N, \\
\\
\displaystyle
f_i(x)^{2-\alpha}\leq K_6\nu(x), & \forall x\in\R^N, \quad i=1,\ldots,N, \\
\\
\displaystyle
h(x)^{2-\beta}\leq K_7\nu(x), & \forall x\in\R^N.
\end{array}
\right.
\label{hyp5:growthDB_2}
\end{align}
\end{description}
\label{hyp:maximum_principle}
\end{hyp}

Under these hypotheses, the Cauchy problem
\[
\left\{
\begin{array}{ll}
D_tu(t,x)=Au(t,x), \quad t>0, & x\in\R^N, \\
\\
u(0,x)=\varphi(x), & x\in\R^N,
\end{array}
\right.
\]
admits a classical solution
\[
u\in C([0,\infty)\times\R^N)\cap C_{loc}^{1+\delta/2,2+\delta}((0,\infty)\times\R^N)
\]
for any $\delta\in(0,1)$ satisfying
\[
\abs*{u(t,x)}\leq\norma\varphi_\infty, \quad t>0,\ x\in\R^N
\]
(see \cite{meta-pallara-wacker}).

If we assume that there exist $\lambda>0$ and a function $f\in C^2(\R^N)$ such
that
\[
\lim_{\abs x\rightarrow+\infty}f(x)=\infty, \quad \sup_{x\in\R^N}(Af(x)-\lambda
f(x))<\infty,
\]
then the classical solution is unique, and we can define a family of bounded
operators $\{S(t)\}_{t\ge0}$ by  $S(t)f(x)=u(t,x)$, for any $t\geq0$, $x\in\R^N$. $\{S(t)\}_{t\geq0}$ is
the contractive semigroup of linear operators associated to the
operator $A$ and, in general, $\{S(t)\}_{t\geq0}$ is neither strongly continuous nor analytic in $C_b(\R^N)$ (see \cite{meta-prio}).

Now we introduce a class of function spaces, which is a natural environment
where to set the Cauchy problem \eqref{eq:eq_Cauchy_NL}:
\begin{defin}
\label{def:K_spaces}
For any $a>0$, let us consider the space
\[
\Kk_a=
\left\{
\begin{array}{l}
h\in C_b\left([T-a,T]\times \R^N\right)\cap
C^{0,1}\left([T-a,T)\times\R^N\right): \\
\\
\displaystyle\sup_{{t\in[T-a,T)}\atop{x\in\R^N}}(T-t)^{1/2}
\abs*{G(x)\nabla h(t,x)}<\infty
\end{array}
\right\},
\]
endowed with the norm
\begin{equation}
\norma h_{\Kk_a}=\norma
h_\infty+[h]_{\Kk_a},
\label{eq:normakalpha}
\end{equation}
where
\[
[h]_{\Kk_a}:=\sup_{t\in[T-a,T)}(T-t)^{1/2}\norma{G\nabla
h(t,\cdot)}_\infty.
\]
\end{defin}

For any $a>0$ we define the function $F_a$ in such a way:
\begin{equation}
F_a:[T-a,T)\times \Kk_a\longrightarrow C(\R^N), \quad
F(t,u)(x)=\psi(x,G(x)\nabla u(t,x)).
\label{eq:def_F}
\end{equation}

Throughout this paper we will write $F$ instead of $F_T$.

At this stage formula
\begin{equation}
v(t,x)=S(T-t)\varphi(x)-\int_t^TS(r-t)F(r,v)(x)dr,
\label{eq:mild_solution}
\end{equation}
is just formal. Since $\psi$ and $G$ may be unbounded, to justify this
formula we need first to show that the semigroup $\{S(t)\}_{t\geq0}$ can actually be applied
to $F$.

\subsection{Weighted gradient estimates}
Our purpose here is to prove that, for any $\varphi\in C_b(\R^N)$ and any $t>0$, the
function $x\mapsto G(x)S(t)\varphi(x)$ is bounded in $\R^N$ and that, for any $T>0$, there exists a
positive constant $C_T$ such that
\[
\norma{G\nabla S(t)\varphi}_\infty\leq\frac{C_T}{t^{1/2}}\norma \varphi_\infty,
\quad t\in(0,T].
\]

For this purpose, for any $R\geq1$, we introduce the function $\eta_R$ defined
by $\eta_R(x)=\eta(\abs x/R)$ for any $x\in\R^N$, where
\[
\eta(t)=
\begin{cases}
1, & t\in[0,1/2], \\
\exp{\left(1-\frac{1}{1-(4t-2)^3}\right)}, & t\in(1/2,3/4), \\
0 & t\geq 3/4.
\end{cases}
\]

Clearly, $\eta_R\in C^2_c(\R^N)$, $0\leq\eta_R\leq 1$ in $\R^N$,
$\eta_R\equiv1$ in
$B(R/2)$, and $\eta_R\equiv 0$ outside
the ball $B(R)$. Moreover, we have
\begin{align}
& D_i\eta_R(x)=-\frac{x_i}{\abs xR}\chi_{[1/2,3/4)}(\abs x/R)\frac{12(4\abs
x/R-2)^2}{\left(1-(4\abs x/R-2)^3\right)^2}\eta_R(x), \label{eq:stima_grad_eta}
\\
& \abs*{\sum_{i=1}^NQ_{ij}(x)D_i\eta_R(x)}
\leq K_8l^j_R(x)\eta_R(x)^{1/3}, \label{eq:stima_Q_grad_eta}\\
\notag \\
& \abs*{\sum_{i=1}^NQ_{ij}(x)D_{ij}\eta_R(x)}
\leq K_9\left(\frac{\prodscal{Q(x)x,x}}{1+\abs x^4}
+\frac{\abs*{Tr[Q(x)]}}{1+\abs x^2}\right), \label{eq:stima_Q_hess_eta}
\end{align}
for any $x\in\R^N$ and any $R\geq1$, where $K_8$ and $K_9$ are positive constant independent of $R$.

\begin{rmk}
In the right-hand side of \eqref{eq:stima_Q_grad_eta} as exponent of $\eta_R$ we could choose any number between $(0,1)$. The exponent $1/3$ is enough to prove the following theorem.
\end{rmk}

\begin{thm}
Let Hypothesis \ref{hyp:maximum_principle} be fulfilled and let $\varphi\in
C_b(\R^N)$. If $u$ is the classical solution
to
the homogenous Cauchy problem
\[
\left\{
\begin{array}{ll}
D_tu(t,x)=Au(t,x), \quad t>0, & x\in\R^N, \\
\\
u(0,x)=\varphi(x), & x\in\R^N,
\end{array}
\right.
\]
i.e., $u\in C_b\left([0,\infty)\times\R^N\right)\cap
C^{1,2}\left((0,\infty)\times\R^N\right)$ and it satisfies the above equation and the initial condition,
then the function
\[
(t,x)\mapsto G(x)\nabla u(t,x)
\]
is bounded in $[\epsilon,T]\times\R^N$, for any $0<\epsilon<T$. Moreover,
there exists a positive constant $C_T$ such that
\begin{equation}
t^{1/2}\norma*{G\nabla u(t,\cdot)}_\infty\leq C_T\norma \varphi_\infty, \quad
\forall
t\in(0,T].
\label{eq:stimadergrad0}
\end{equation}
\label{thm:epsilon_Bounded_C}
\end{thm}

\begin{proof}[Proof]
Fix $R\geq1$, $T>0$ and let $u_R\in C_b\left([0,\infty)\times \overline{B(R)}\right)\cap
C^{1,2}\left((0,\infty)\times\overline{B(R)}\right)$ be the solution to the Cauchy Dirichlet problem
\begin{align}
\left\{
\begin{array}{lll}
D_tu_R(t,x)=Au_R(t,x), & t>0, & x\in B(R), \\
\\
u_R(t,x)=0, & t>0, & x\in\partial B(R), \\
\\
u_R(0,x)=\eta_R(x)\varphi(x), & & x\in \overline{B(R)}.
\end{array}
\right.
\label{eq:eq_Cauchy_Dirichlet}
\end{align}
We set
\[
v_R(t,x)=u_R(t,x)^2+at\eta_R^2\abs{G(x)\nabla u_R(t,x)}^2, \quad
t\in[0,T], \quad
x\in \overline{B(R)}.
\]

Function $v_R$ is continuous in its domain, and it solves the Cauchy problem
\begin{align}
\left\{
\begin{array}{lll}
D_tv_R(t,x)-Av_R(t,x)=g_R(t,x), & t\in[0,T], & x\in B(R), \\
\\
v_R(t,x)=0, & t\in[0,T], & x\in\partial B(R), \\
\\
v_R(0,x)= (\eta_R\varphi)^2(x), & & x\in \overline{B(R)},
\end{array}
\right.
\label{eq:eq_Cauchy_Dirichlet2}
\end{align}
where $g_R(t,x)=t\displaystyle\sum_{i=1}^6g_{i,R}(t,x)$ and
\begin{align*}
g_{1,R}
& = -2t^{-1}\abs{G\nabla
u_R}^2-2a\eta_R^2\sum_{i,j=1}^NQ_{ij}\prodscal{G\nabla
(D_iu_R),G\nabla (D_ju_R)} \\
& \quad -2a\eta_R \sum_{i=1}^NB_i D_i\eta_R\abs{G\nabla u_R}^2, \\
\\
g_{2,R}
& = 2a\eta_R^2\prodscal{G(DB)\nabla u_R,G\nabla u_R}-2a\eta_R^2\sum_{i,j=1}^NQ_{ij}\prodscal{(D_{ij}G)\nabla
u_R,G\nabla u_R} \\
\qquad & -2a\eta_R^2\sum_{j=1}^NB_j\prodscal{(D_jG)\nabla
u_R,G\nabla
u_R}, \\
\\
g_{3,R}
& = -2a\abs{G\nabla u_R}^2\abs{G\nabla
\eta_R}^2-2a\eta_R^2\sum_{i,j=1}^NQ_{ij}\prodscal{(D_jG)\nabla
u_R,(D_iG)\nabla u_R}, \\
\\
g_{4,R}
& = - 2a\eta_R Tr[Q(D^2\eta_R)]\abs{G\nabla u_R}^2-8a\eta_R
\sum_{i,j=1}^NQ_{ij}(D_i\eta_R)\prodscal{(D_jG)\nabla u_R,G\nabla u_R}
\\
\qquad & - 8a\eta_R \sum_{i,j=1}^NQ_{ij}(D_i\eta_R)\prodscal{G\nabla
(D_ju_R),G\nabla u_R}, \\
\\
g_{5,R}
& = -4a\eta_R^2\sum_{i,j=1}^NQ_{ij}\prodscal{(D_jG)\nabla
(D_iu_R),G\nabla u_R}-4a\eta_R^2\sum_{i,j=1}^NQ_{ij}\prodscal{(D_jG)\nabla
u_R,G\nabla (D_iu_R)} \\
\qquad & + 4a\eta_R^2\prodscal{GTr[(\nabla
G)G(D^2u_R)],G\nabla u_R}, \\
\\
g_{6,R}
& = 2a\eta_R^2\abs{G\nabla u_R}^2.
\end{align*}

We are going to prove that there exists a positive constant $K$, independent of
$R$, such that $g_R(t,x)\leq Kv_R(t,x)$, for any $(t,x)\in [0,T]\times B(R)$.

The terms $g_{1,R}$ and $g_{2,R}$ are crucial in the estimate of $g_R$, since
they allow us to control all the other terms $g_{i,R}$, $i=3,4,5,6$.

Using \eqref{hyp0:DB} in Hypothesis \ref{hyp:maximum_principle}, we get
\[
\begin{split}
g_{1,R}(t,x)
& \leq-2t^{-1}\abs{G\nabla u_R}^2-2a\eta_R^2\nu(x)\sum_{i=1}^N\abs*{G\nabla
(D_iu_R)}^2
-2a\eta_R \prodscal{B,\nabla\eta_R}\abs{G\nabla u_R}^2, \\
g_{2,R} & =2a\eta_R^2\prodscal{M G\nabla u_R,G\nabla u_R}\leq
-2a\eta_R^2b(x)\abs*{G\nabla u_R}^2, \\
\\
g_{3,R}& \leq 0.
\end{split}
\]

$g_{4,R}$ is the awkward term. We have to pay particular attention to the way we
estimate its addends which we want to compare with $g_{1,R}$ and $g_{2,R}$.

As far as the first addend is concerned, taking advantage of
\eqref{eq:stima_Q_hess_eta} and of the well known Young's inequality
$ab\leq(\epsilon/2) a^2+(2\epsilon)^{-1}b^2$, which holds true for any
$a,b,\epsilon>0$, by \eqref{eq:stima_Q_hess_eta} we get
\[
\begin{split}
& \abs*{2a\eta_R\sum_{i,j=1}^NQ_{ij}(D_{ij}\eta_R)\abs*{G\nabla u_R}^2} \\
& \leq \frac a\epsilon\abs*{G\nabla u_R}^2
+a\epsilon\eta_R^2\abs*{\sum_{i,j=1}^NQ_{ij}(D_{ij}\eta_R)}^2\abs*{G\nabla
u_R}^2 \\
& \leq \frac a\epsilon\abs*{G\nabla u_R}^2
+2K_9a\epsilon\eta_R^2\left(\left(\frac{\langle Q(x)x,x\rangle}{1+\abs
x^4}\right)^2
+\left(\frac{\abs*{Tr[Q(x)]}}{1+\abs
x^2}\right)^2\right)\abs*{G\nabla u_R}^2.
\end{split}
\]
As far as the second term in the definition of $g_{4,R}$ is concerned, we have
\[
\begin{split}
& \abs*{8a\eta_R \sum_{i,j=1}^NQ_{ij}(D_i\eta_R)\prodscal{\left(D_jG\right)\nabla
u_R,G\nabla u_R}}\\
& = \abs*{8a\eta_R
\sum_{i,j=1}^NQ_{ij}(D_i\eta_R)\prodscal{\left(D_jG\right)G^{-1}G\nabla
u_R,G\nabla u_R}}\\
& \leq
8a\eta_R\sum_{i=1}^N\abs*{D_i\eta_R}\abs*{\sum_{j=1}^NQ_{ij}\left(D_jG\right)G^
{ -1 } }
\abs*{G\nabla u_R}^2 \\
& = 8a\eta_R\sum_{i=1}^N\abs*{D_i\eta_R} f_i\abs*{G\nabla u_R}^2.
\end{split}
\]

The last term in the definition of $g_{4,R}$ is the worst one because we need
to estimate the growths of both $\abs*{G\nabla u_R}$ and $\abs*{G\nabla
D_ju_R}$. We split it using the following inequality, which follows applying
twice the Young's inequality, and holds for any $A,B,C,\epsilon>0$:
\[
ABC\leq\frac 1 4 \left(2\epsilon C^2+\frac1\epsilon A^4+\frac 1
\epsilon B^4\right).
\]

We set
\begin{align*}
A & =
a^{3/8}\eta_R^{\delta/6}\abs*{\sum_{i=1}^NQ_{ij}D_i\eta_R}^{1-\delta/2}
\abs*
{ G \nabla u_R}^{1/2},\\
 B  &= a^{1/8}\abs*{G\nabla u_R}^{1/2}, \\
C & =
a^{1/2}\eta_R^{1-\delta/6}\abs*{\sum_{i=1}^NQ_{ij}D_i\eta_R}^{\delta/2}
\abs*
{G\nabla (D_ju_R)},
\end{align*}
where $\delta$ is defined in \eqref{hyp1:growth_Q}, and recall that
\[
l^j_R(x) =\frac{1}{1+R^2}\abs*{\sum_{i=1}^NQ_{ij}(x)x_i},
\]
for any $x\in\R^N$, $R\geq1$, $j=1,\ldots,N$. The particular split into $A, B$ and $C$ arises from the necessity of having coefficients of $\abs*{G\nabla (D_ju_R)}^2$, $j=1,\ldots,N$, and of $\abs*{G\nabla u_R}^2$ which we can estimate with $g_{1,R}$ and $g_{2,R}$. \eqref{eq:stima_Q_grad_eta} and straightforward computations yield
\[
\begin{split}
&
\abs*{8a\eta_R\sum_{i,j=1}^NQ_{ij}(D_i\eta_R)\prodscal{G\nabla
(D_ju_R),G\nabla u_R}}\notag \\
& \quad \leq
8\sum_{j=1}^N\left(a^{1/2}\eta_R^{1-\delta/6}|(Q\nabla\eta_R)_j|^{
\delta/2}\abs*{G\nabla (D_ju_R)}\right. \notag \\
& \qquad \left.\times
a^{3/8}\eta_R^{\delta/6}|(Q\nabla\eta_R)_j|^{1-\delta/2}
\abs*{G\nabla u_R}^{1/2}a^{1/8}\abs*{G\nabla u_R}^{1/2} \right) \notag \\
& \quad \leq
4aK_8^\delta\epsilon\eta_R^{2-\delta/3}\sum_{j=1}^N\left(l^j_R\right)^\delta\eta_R^{\delta/3}\abs*{G\nabla (D_ju_R)}^2\notag \\
& \qquad +\frac{2a^{3/2}K_8^{3-2\delta}}\epsilon
\eta_R\sum_{j=1}^N|(Q\nabla\eta_R)_j|
\left(l_R^j\right)^{3-2\delta}\abs*{G\nabla u_R}^2+\frac{2a^{1/2}N}{\epsilon}\abs*{G\nabla u_R}^2, \notag
\end{split}
\]
where we have estimated $|(Q\nabla\eta_R)_j|^{3-2\delta}$  by \eqref{eq:stima_Q_grad_eta} and we have kept the factor $|(Q\nabla\eta_R)_j|$ since we want as coefficient
\[
\frac{1}{\abs xR}\chi_{(1/2,3/4)}(\abs x/R)\frac{12(4\abs
x/R-2)^2}{\left(1-(4\abs x/R-2)^3\right)^2}.
\]

Hence we get
\[
\begin{split}
& \quad \leq 4aK_1K_8^\delta\epsilon\eta_R^2\nu\sum_{j=1}^N\abs*{G\nabla (D_ju_R)}^2
+\frac{2a^{3/2}K_8^{3-2\delta}}\epsilon\eta_R
\sum_{j=1}^N|(Q\nabla\eta_R)_j|\left(l^j_R\right)^{3-2\delta}\abs*{G\nabla u_R}^2 \notag\\
& \qquad +\frac{2Na^{1/2}}{\epsilon}\abs*{G\nabla u_R}^2 \notag\\
& \quad \leq 4aK_1K_8^\delta\epsilon\eta_R^2\nu\sum_{j=1}^N\abs*{G\nabla
(D_ju_R)}^2 \notag \\
& \qquad +\frac{2a^{3/2}K_8^{3-2\delta}}\epsilon\eta_R^2\frac{1}{\abs xR}\chi_{(1/2,3/4)}(\abs x/R)\frac{12(4\abs
x/R-2)^2}{\left(1-(4\abs x/R-2)^3\right)^2} \notag \\
& \qquad \times\sum_{j=1}^N\Big|\sum_{i=1}^NQ_{ij}x_i\Big|\left(l^j_R\right)^{3-2\delta}
\abs*{G\nabla u_R}^2 \notag\\
& \qquad +\frac{2Na^{1/2}}{\epsilon}\abs*{G\nabla u_R}^2,
\end{split}
\]

The last term that we need to estimate is $g_{5,R}$. Applying the Young inequality with $alpha$ and $\beta$ as in \eqref{hyp5:growthDB_2} we get
\begin{align*}
& \abs*{4a\eta_R^2\sum_{i,j=1}^NQ_{ij}\prodscal{\left(D_jG\right)G^{-1}G\nabla
(D_iu_R),G\nabla u_R}} \\
& \quad
\leq\frac{2a}\epsilon\eta_R^2\sum_{i=1}^N\abs*{\sum_{j=1}^NQ_{ij}
\left(D_jG\right)G^ {
-1}}^{\alpha}
\abs*{G\nabla u_R}^2 \\
& \qquad
+2a\epsilon\eta_R^2\sum_{i=1}^N\abs*{\sum_{j=1}^NQ_{ij}\left(D_jG\right)G^{-1}}
^ {
2-\alpha}
\abs*{G\nabla (D_iu_R)}^2, \\
& \quad = \frac{2a}\epsilon\eta_R^2\abs*{G\nabla
u_R}^2\sum_{i=1}^Nf_i(x)^\alpha +2a\epsilon\eta_R^2\abs*{G\nabla
(D_iu_R)}^2\sum_{i=1}^Nf_i(x)^{2-\alpha}, \\
\\
& \abs*{4a\eta_R^2\sum_{i,j=1}^NQ_{ij}\prodscal{\left(D_jG\right)G^{-1}G\nabla
u_R,G\nabla (D_iu_R)}} \\
& \quad \leq
\frac{2a}{\epsilon}\eta_R^2\abs*{G\nabla
u_R}^2\sum_{i=1}^Nf_i(x)^{\alpha}
+2a\epsilon\eta_R^2\abs*{G\nabla (D_iu_R)}^2\sum_{i=1}^Nf_i(x)^{2-\alpha},
\\
\\
& \abs*{4a\eta_R^2\prodscal{GTr[\left(\nabla G\right) G(D^2u_R)],G\nabla u_R}} \\
& \quad = 4a\eta_R^2
\Big|\sum_{i,j,l,m=1}^N G_{ij}D_jG_{lm}(G\nabla D_j u_R)_m(G\nabla u_R)_i\Big| \\
& \quad \leq 4a\eta_R^2\sum_{i,j,l,m=1}^N\left[\Big|G_{ij}D_jG_{lm}\Big|^{1-\beta/2}|(G\nabla D_ju_R)_m|\right]
\left[\Big|G_{ij}D_jG_{lm}\Big|^{\beta/2}|(G\nabla u_R)_i|\right] \\
& \quad \leq
\frac{2a}\epsilon\eta_R^2\sum_{i,j,l,m=1}^N\abs*{G_{ij}D_jG_{lm}}^
{\beta}
\abs*{G \nabla u_R}^2\\
&
\qquad
+2a\epsilon\eta_R^2\sum_{i,j,l,m=1}^N\abs*{G_{ij}D_jG_{lm}}^{
2-\beta}
\sum_{i=1}^N\abs*{G\nabla (D_iu_R)}^2 \\
& \quad = \frac{2a}\epsilon\eta_R^2h(x)^{\beta}\abs*{G \nabla u_R}^2
+2a\epsilon\eta_R^2h(x)^{2-\beta}\sum_{i=1}^N\abs*{G\nabla (D_iu_R)}^2.
\end{align*}

Hence, collecting the similar terms, and recalling that
\[
D_i\eta_R(x)=-\frac{x_i}{\abs xR}\chi_{(1/2,3/4)}(\abs x/R)\frac{12(4\abs
x/R-2)^2}{\left(1-(4\abs x/R-2)^3\right)^2}\eta_R(x),
\]
we deduce that
\[
g_R(t,x)\leq I_1(t,x)\abs {G(x)\nabla u_R(t,x)}^2+\sum_{i=1}^NI_{2,i}\abs{G(x)\nabla
(D_iu_R)(t,x)}^2
\]
for any $t\in(0,T]$, $x\in\R^N$, where
\begin{align}
I_1(t,x)
& = \left(-2+2a+\frac{2a^{1/2}Nt}{\epsilon}+\frac{at}\epsilon\right) \label{eq:stimaderu1}\\
& \quad +2at\eta_R(x)^2\chi_{\left[\frac12,\frac34\right)}(\abs
x/R)\frac{12(4\abs x/R-2)^2}{\abs xR\left(1-(4\abs x/R-2)^3\right)^2} \notag \\
& \quad \times\left(\sum_{i=1}^Nx_i B_i(x)
+4\sum_{i=1}^N\abs*{x_i}f_i(x)
+\frac{K_8^{3-2\delta}a^{1/2}}\epsilon\sum_{j=1}^N\abs*{(Qx)_j}\left(l^j_R(x)\right)
^{3-2\delta}\right) \label{eq:stimaderu2} \\
& \quad +2at\eta_R^2(x)\left(\frac2\epsilon\sum_{i=1}^Nf_i(x)^{\alpha}
+\frac 1\epsilon h(x)^{\beta}\right. \notag \\
& \quad \left.+K_9\epsilon\left[\left(\frac{\prodscal{Q(x)x,x}}{1+\abs
x^4}\right)^2+\left(\frac{\abs*{Tr[Q(x)]}}{1+\abs x^2}\right)^2\right] -b(x)\right), \label{eq:stimaderu3} \\
& &\notag \\
I_{2,i}(t,x)
& = 2at\eta_R^2\left(-\nu(x)+2K_1K_8^\delta\epsilon\nu(x)+2\epsilon
f_i(x)^{2-\alpha}+\epsilon h(x)^{2-\beta}\right) \notag \\[1mm]
& \leq 2at\eta_R^2\nu(x)\left(-1+2K_1K_8^\delta\epsilon+2K_6\epsilon+K_7\epsilon\right).
\label{eq1:stimadergradu1}
\end{align}

We now choose the parameters $a,\epsilon,n$ to guarantee that $I_1(t,x)$ is bounded
from above and $I_{2,i}(t,x)\leq0$ for any $t\in(0,T]$, $x\in\R^N$, $i=1,\ldots,N$. The choice for $I_{2,i}$ is immediate; indeed, it is easy to see that the right-hand side in \eqref{eq1:stimadergradu1} is non positive if and only if we choose $\epsilon>0$ such that
\[
-1+\left(2K_1K_8^\delta+2K_6+K_7\right)\epsilon\leq0.
\]
In such a way, the coefficients of $\abs{G\nabla (D _iu_R)}^2$ are negative,
for any $i$.

Now we consider $I_1$; it is bounded from above if and only if all the terms in the brackets in \eqref{eq:stimaderu1}, \eqref{eq:stimaderu2} and \eqref{eq:stimaderu3} are bounded. At first we find condition on $\epsilon$ such that \eqref{eq:stimaderu3} is bounded; by \eqref{hyp4:growthDB_1} we can easily deduce that \eqref{eq:stimaderu3} is smaller than
\[
\frac2\epsilon\sum_{i=1}^Nf_i(x)^{\alpha}
+\frac 1\epsilon h(x)^{\beta}+(\epsilon K_9/K_4-1)b(x)+K_5/K_4.
\]

This function has the same form of the left-hand side in \eqref{hyp5:growthDB_2}, hence for any $n\in\N$, which satisfies
\[
n\geq\frac{2}{\epsilon(1-K_9\epsilon/K_4)},
\]
\eqref{eq:stimaderu3} is bounded from above.

Fixed $\epsilon$, we get an estimate from above of \eqref{eq:stimaderu2} provided the following condition on $a$ is satisfied:
\[
a^{1/2}\leq\frac {K_2\epsilon}{K_8}.
\]

Finally, \eqref{eq:stimaderu1} is bounded.

With the previous choices of the parameters, $I_1$ turns out to be bounded from
above.

From \eqref{hyp3:growthB}, \eqref{hyp4:growthDB_1}, \eqref{hyp5:growthDB_2} we
obtain that $g(t,x)\leq cv_R(t,x)$, for any $(t,x)\in[0,T]\times B(R)$ and some $c>0$.
Hence, $v_R$ satisfies
\[
\left\{
\begin{array}{lll}
D_tv_R(t,x)-Av_R(t,x)\leq cv_R(t,x), & t\in(0,T], & x\in B(R), \\
\\
v_R(t,x)=0, & t\in[0,T], & x\in \partial B(R), \\
\\
v_R(0,x)=(\eta_R\varphi)^2(x), & & x\in \overline{B(R)}.
\end{array}
\right.
\]

The classical maximum principle shows that
\[
\abs{v_R(t,x)}\leq K \norma {\eta_R\varphi}^2_\infty\leq K\norma
\varphi^2_\infty, \quad t\in[0,T], \ x\in\overline{B(R)},
\]
for some positive constant $K$ independent of $R$. Taking the limit as
$R\rightarrow\infty$, we deduce that the function
$v(t,x)=u(t,x)^2+at\abs{G(x)\nabla u(t,x)}^2$ satisfies
\[
\abs{v(t,x)}\leq K\norma \varphi^2_\infty,
\]
so that the statement is proved.
\end{proof}

\begin{rmk}
By the semigroup property, it easily follows that, for any $\omega>0$, there
exists $C=C(\omega)>0$ such that
\begin{equation}
\norma*{G\nabla S(t)\varphi}_\infty\leq \frac{Ce^{\omega t}}{t^{1/2}}\norma
\varphi_\infty,
\label{eq:smgr_property}
\end{equation}
for any $t>0$ and any $\varphi\in C_b(\R^N)$.

Indeed, for any $\omega>0$, we can choose $\sigma=\sigma(\omega)$ such that
$e^{\omega t}t^{-1/2}>1$, for any $t>\sigma$. If $t>\sigma$ we can estimate
(using \eqref{eq:stimadergrad0} and recalling that $\{S(t)\}_{t\geq0}$ is a contraction
semigroup)
\[
\begin{split}
\norma*{G\nabla S(t)\varphi}_\infty
& =\norma*{G\nabla S(\sigma)S(t-\sigma)\varphi}_\infty \leq
\frac{C_\sigma}{\sigma^{1/2}}\norma*{S(t-\sigma) \varphi}_\infty \\
& \leq \frac{C_\sigma}{\sigma^{1/2}}\norma \varphi_\infty \leq
\frac{C_\sigma e^{\omega
t}}{\sigma^{1/2}t^{1/2}}\norma \varphi_\infty,
\end{split}
\]
and therefore \eqref{eq:smgr_property} holds with
$C=\max\{C_\sigma,\sigma^{-1/2}C_\sigma\}$.
\end{rmk}

Now we provide a class of operators $A$ which satisfy Hypothesis
\ref{hyp:maximum_principle}.
\begin{ex}
\label{ex:coeffA1}
Let $Q,B$ be defined as follows:
\[
Q_{ij}(x)=q_{ij}(1+\abs x^2)^m, \quad B_i(x)=-b_ix_i(1+\abs x^2)^p, \quad
\forall x\in\R^N,
\]
where $m,p>0$, $b_i>0$ for any $i=1,\ldots,N$, and $q=q_{ij}$ is a
positive-definite matrix such that
\[
\prodscal{q\xi,\xi}\geq\nu_0\norma\xi^2, \quad \forall \xi\in\R^N.
\]

If $N\geq2$, condition \eqref{hyp0:DB} is satisfied if and only if
\[
m\leq\frac b B,
\]
where $b=\min\{b_i\}$, $B=\max\{b_i\}$. With this restriction, in
\eqref{hyp1:growth_Q} it is possible to choose $\delta=3/2$, and conditions
\eqref{hyp3:growthB}, \eqref{hyp4:growthDB_1} and \eqref{hyp5:growthDB_2} are
fulfilled for any $p>m-1$.

If $N=1$, \eqref{hyp0:DB} is satisfied if $2p+1>m$, and, to satisfy also
\eqref{hyp1:growth_Q}, it is necessary to take $\delta\in[0,3/2]$ such that
$\delta\geq2m(\delta-1)$. One can easily check that there exists
$\delta\geq1$ which satisfies the previous inequality, and,
consequently, if $p>m$, then even \eqref{hyp3:growthB}, \eqref{hyp4:growthDB_1}
and \eqref{hyp5:growthDB_2} are fulfilled.
\end{ex}

\begin{prop}
Under the same assumptions of Theorem \ref{thm:epsilon_Bounded_C}, if $\varphi\in
C^1_b(\R^N)$, then the function
\[
(t,x)\mapsto G(x)\nabla S(t)\varphi(x)
\]
is bounded in $[0,T]\times \R^N$.
\label{prop:Bounded_C1}
\end{prop}

\begin{proof}
The proof is quite similar to the one of Theorem \ref{thm:epsilon_Bounded_C},
hence we just sketch it. We fix $R\geq1$, and denote by $u_R$ the solution to the
Dirichlet Cauchy problem \eqref{eq:eq_Cauchy_Dirichlet}. Further we set
\[
v_R(t,x)=u_R(t,x)^2+a\eta_R^2\abs{G(x)\nabla u_R(t,x)}^2, \quad
t\in(0,T], \quad
x\in \overline{B(R)}.
\]

Function $v_R$ is continuous in its domain and it solves the Cauchy problem
\[
\left\{
\begin{array}{lll}
D_tv_R(t,x)-Av_R(t,x)=\tilde g_R(t,x), & t\in[0,T], & x\in B(R), \\
\\
v_R(t,x)=0, & t\in[0,T], & x\in\partial B(R), \\
\\
v_R(0,x)= (\eta_R\varphi)^2(x), & & x\in \overline{B(R)},
\end{array}
\right.
\]
where $\tilde g_R(t,x)=\tilde g_{1,R}(t,x)+\displaystyle\sum_{i=2}^5
g_{i,R}(t,x)$,
\begin{align*}
\tilde g_{1,R}
& = -2\abs{G\nabla
u_R}^2-2a\eta_R^2\sum_{i,j=1}^NQ_{ij}\prodscal{G\nabla
(D_iu_R),G\nabla (D_ju_R)} \\
& \quad -2a\eta_R \sum_{i=1}^NB_iD_i\eta_R\abs{G\nabla u_R}^2,
\end{align*}
and $g_{i,R}$, $i=2,3,4,5$, have been defined in Theorem
\ref{thm:epsilon_Bounded_C}. Repeating the computations of Theorem
\ref{thm:epsilon_Bounded_C}, we see that
\[
\tilde g_R\leq I_1\abs {G\nabla u_R}^2+\sum_{i=1}^NI_{2,i}\abs{G\nabla
(D_iu_R)}^2,
\]
where
\[
\begin{split}
I_1
& = \left(-2+a\eta_R+\frac{2a^{1/2}N}{\epsilon}+\frac{a}\epsilon\right) \\
& \quad +2a\eta_R(x)^2\chi_{\left(\frac12,\frac34\right)}(\abs
x/R)\frac{12(4\abs x/R-2)^2}{\abs xR\left(1-(4\abs x/R-2)^3\right)^2}
\\
& \quad \times\left(\sum_{i=1}^Nx_iB_i(x)
+4\sum_{i=1}^N\abs*{x_i}f_i(x)
+\frac{K_8^{3-2\delta}a^{1/2}}\epsilon\sum_{j=1}^N\abs*{(Qx)_j}\left(l^j_R(x)\right)^{3-2\delta}\right) \\
& +2a\eta_R^2\left(\frac2\epsilon\sum_{i=1}^Nf_i(x)^{\alpha}
+\frac 1\epsilon
h(x)^{\beta}\right. \\
& \quad+\left.K_9\epsilon\left[\left(\frac{\abs*{\prodscal{Q(x)x,x}}}{1+\abs
x^4}\right)^2
+\left(\frac{\abs*{Tr[Q(x)]}}{1+\abs x^2}\right)^2\right] -b(x)\right), \\
\\
I_{2,i}
& = 2a\eta_R^2\left(-\nu(x)+2K_1K_8\epsilon\nu(x)+2\epsilon
f_i(x)^{2-\alpha}+\epsilon h(x)^{2-\beta}\right) \\[1mm]
& \leq2a\eta_R^2\nu(x)\left(-1+2K_1K_8\epsilon+2K_6\epsilon+K_7\epsilon\right).
\end{split}
\label{eq1:stimadergradu1}
\]

A suitable choice of the parameters $a,\epsilon,n$ guarantees that
$I_2\leq0$ and that there exists a positive constant $C$ such that $I_1\leq C$.
Hence, $v_R$ satisfies
\[
\left\{
\begin{array}{lll}
D_tv_R(t,x)-Av_R(t,x)\leq Cv_R(t,x), & t\in[0,T], & x\in B(R), \\
\\
v_R(t,x)=0, & t\in[0,T], & x\in\partial B(R), \\
\\
v_R(0,x)= (\eta_R\varphi)^2(x), & & x\in \overline{B(R)},
\end{array}
\right.
\]
and the classical maximum principle shows that
\[
\abs*{v_R(t,x)}\leq K\norma\varphi_\infty, \quad (t,x)\in[0,T]\times\overline{B(R)},
\]
for some positive constant $K=K(T)$ independent of $R$. Taking the limit as
$R\rightarrow+\infty$, the assertion follows.
\end{proof}

\subsection{Existence and uniqueness of a mild solution to the problem
\eqref{eq:eq_Cauchy_NL}}
In this part we will prove that the operator $\Gamma$ defined for any
$u\in\Kk_T$ by
\[
(\Gamma u)(t,x):=S(T-t)\varphi(x)-\int_t^TS(r-t)F(r,u)(x)dr, \quad \forall t\in[0,T], \quad x\in\R^N,
\]
admits a unique fixed point. We call a {\it mild solution} of problem \eqref{eq:eq_Cauchy_NL} any fixed point $v\in\Kk_T$ of the operator $\Gamma$.

\begin{rmk}
If $\psi$ satisfies Hypothesis \ref{hyp:g_lipschitz}, then (see
\eqref{eq:def_F})
\begin{equation}
\begin{split}
(i) \ & \ \norma{F(s,u)-F(s,v)}_\infty \leq L_\psi (T-s)^{-1/2}[u-v]_{\Kk_T},
\quad
s\in[0,T),x\in\R^N, \\
(ii) \ & \ \norma{F(s,u)}_\infty \leq
L_\psi\left(1+(T-s)^{-1/2}[u]_{\Kk_T}\right),
\end{split}
\label{eq:F_lipschitz}
\end{equation}
for any $u,v\in \Kk_T$. Moreover, if $u\in \Kk_T$,
$F(\cdot,u)(\cdot):[0,T)\times \R^N\longrightarrow \R^N$ belongs to
$C([0,T)\times\R^N)$.
\label{rmk:g_lipschitz}
\end{rmk}

The following proposition shows some continuity and boundless properties of the functions which belong to $\Kk_a$, for some $a>0$.

\begin{prop}
\label{prop:regularity_mild_solution}
If $u \in \Kk_a$, for some $a>0$, $F$ satisfies \eqref{eq:F_lipschitz} and
\[
\sup_{t\in(T-a,T)}(T-t)^{1/2}\norma*{G\nabla u(t,\cdot)}_\infty<\infty,
\]
then the functions
\[
(t,x)\mapsto \tilde F(t,x):=\int_t^TS(r-t)F(r,u)(x)dr
\]
and
\[
(t,x)\mapsto G(x)\nabla\tilde F(t,x)
\]
are continuous and bounded in $[T-a,T]\times\R^N$.
\end{prop}

\begin{proof}[Proof]
For any $t\in[T-a,T]$, the functions
\[
x\mapsto \tilde F(t,x):=\int_t^TS(r-t)F(r,u)(x)dr, \quad x\mapsto G(x)\nabla\tilde F(t,x)
\]
are continuous in $\R^N$. Hence it is enough to show that these functions are continuous with respect to $t$, locally uniformly with respect to $x$.

Let $(t_0,x_0) \in(T-a,T)\times \R^N$, $B=B(x_0,1)\in\R^N$, and fix
$t\in(t_0-\delta,t_0+\delta)$, where $0<\delta<\min\{T-t_0,a+t_0-T\}$.
We will only prove the continuity from the right with respect to time, uniformly with respect to $x$, since the continuity from the left can be proved arguing in
the same way. Hence we consider $t\in(t_0,t_0+\delta)$. We have
\[
\begin{split}
\abs*{\tilde F(t_0,x)-\tilde F(t,x)}
& \leq \int_t^T\abs*{S(r-t_0)F(r,u)(x)-S(r-t)F(r,u)(x)}dr
\\
& \quad +\int_{t_0}^t\abs*{S(r-t_0)F(r,u)(x)}dr \\
& =
\int_{t_0}^T\abs*{S(r-t_0)F(r,u)(x)-S(r-t)F(r,u)(x)}\chi_{(t,
T)}(r)dr \\
& \quad +\int_{t_0}^t\abs*{S(r-t_0)F(r,u)(x)}dr \\
& =:I_1(t,x)+I_2(t,x).
\end{split}
\]

Since $\norma*{S(r-t_0)F(r,u)}_\infty\leq C$, for any $r\in(t_0,t_0+\delta)$,
$I_2$ tends to $0$, as $t$ tends to $t_0$, uniformly with respect to $x\in B$.

Now we consider $I_1$. Since $u\in \Kk_a$, we can estimate the function under the integral sign
as follows:
\[
\begin{split}
\norma*{S(r-t_0)F(r,u)-S(r-t)F(r,u)}_\infty\chi_{(t,T)}(r)
& \leq 2M_0\norma*{F(r,u)}_\infty \\
& \leq 2M_0L_\psi\left(1+(T-r)^{-1/2}[u]_{\Kk_a}\right),
\end{split}
\]
for any $r\in(t_0,T)$, and the last function is integrable in $(t_0,T)$.

Finally, for any $r\in(0,T)$, $F(r,u)\in C_b(\R^N)$ by \eqref{eq:F_lipschitz}. Hence
$S(\cdot)F(r,u)(\cdot)$ belongs to
$C([0,\infty)\times\R^N)$, and
\[
\lim_{t\downarrow t_0}\abs{S(r-t_0)F(r,u)(x)-S(r-t)F(r,u)(x)}=0,
\]
uniformly with respect to $x\in B$, for any $r\in(0,T)$.


By dominated convergence we can conclude that $I_1$ tends to $0$ as
$t$ approaches $t_0$, uniformly with respect to $x\in B$.

Proving the continuity of the gradient is a bit more complicated. Let
$t_0,x_0,t,B,\delta$ be as
above; we have
\[
\begin{split}
&\abs*{G(x)\nabla\tilde F(t_0,x)-G(x)\nabla\tilde
F(t,x)} \\
& \quad \leq \int_t^T\abs*{G(x)\nabla
S(r-t_0)F(r,u)(x)-G(x)\nabla S(r-t)F(r,u)(x)}dr \\
& \qquad +\int_{t_0}^t\abs*{G(x)\nabla S(r-t_0)F(r,u)(x)}dr \\
& \qquad =:\tilde I_1(t,x)+\tilde I_2(t,x).
\end{split}
\]

By Theorem \ref{thm:epsilon_Bounded_C}, there exists a positive constant C such
that
\[
\norma*{G\nabla S(r-t_0)F(r,u)}_{\infty}\leq
(r-t_0)^{-1/2} C,
\]
for any $r\in(t_0,t_0+\delta)$. Hence $\tilde I_2$ tends to zero as $t$ tends
to $t_0$, uniformly with respect to $x\in B$.

The term $\tilde I_1$ should be analyzed differently. Fix $\epsilon>0$ and
$t\in(t_0,t_0+\delta)$ such that $t-t_0<\epsilon^2$. Now we
split the integral:
\[
\begin{split}
&\tilde I_1(t,x) \\
& \quad =\int_{t_0+\epsilon^2}^{T}\abs*{G(x)\nabla
S(r-t_0)F(r,u)(x)-G(x)\nabla S(r-t)F(r,u)(x)}dr \\
& \quad +\int_{t_0}^{t_0+\epsilon^2}\abs*{G(x)\nabla
S(r-t_0)F(r,u)(x)-G(x)\nabla S(r-t)F(r,u)(x)}\chi_{(t,T)}(r)dr \\
& \quad =:J_1(t,x)+J_2(t,x).
\end{split}
\]

Easy computations show that there exists a positive constant $C>0$,
independent of $t,x$, such that
\[
J_2(t,x)\leq C \epsilon, \quad \forall t\in(t_0,t_0+\epsilon^2),\quad \forall
x\in
B.
\]

For $J_1$, it is enough to observe that the function under the integral
sign converges to $0$ pointwise with respect to $t$, locally uniformly with
respect to $x$, and that the function $h$, defined by
\[
h(r)=C_TL_\psi\left(1+(T-r)^{-1/2}[u]_{\Kk_T}\right)\left((r-t_0)^{-1/2}+(r-t_0-\epsilon^2)^{-1/2}\right)
\]
is independent on $t$ and $x$ and bounds $J_1$ from above. Dominated convergence
allows us to conclude that $J_1(t,x)$ vanishes to $0$ as $t$ tends to $t_0$,
locally uniformly with respect to $x$. Hence, there exists
$c_\epsilon\leq\epsilon^2$ such that, if $t_0-t<c_\epsilon$ and $x\in B$, then
$J_1(t,x)\leq \epsilon$. It means that there exists a suitable $C>0$ such that
$\tilde I_1(t,x)\leq C\epsilon$ for any $t>t_0-c_\epsilon$ and $x\in B$.
\end{proof}

We now look for a solution to problem \eqref{eq:eq_Cauchy_NL} in $\Kk_T$. At
first, we show that, if $u$ is a mild
solution of \eqref{eq:eq_Cauchy_NL} in $\Kk_a$, for some $a\in(0,T)$, then it is the unique mild
solution in such a space.

\begin{prop}[Uniqueness]
If problem \eqref{eq:eq_Cauchy_NL} admits a mild solution in $\Kk_a$, then it is unique.
\label{prop:uniqueness_1}
\end{prop}

\begin{proof}[Proof]
Let $u,v\in \Kk_a$ be two mild solutions of \eqref{eq:eq_Cauchy_NL}. Then,
taking \eqref{eq:g_lipschitz} and \eqref{eq:stimadergrad0} into account, for any
$t\in[T-a,T]$ we get

\[
\begin{split}
\norma{G\nabla(u-v)(t,\cdot)}_\infty
&  \leq \norma*{\int_t^TG\nabla S(r-t)\left(F(r,u)-F(r,v)\right)dr}_\infty \\
& \leq C_TL_\psi\int_t^T(r-t)^{-1/2}\norma{G\nabla(u-v)(r,\cdot)}_\infty dr \\
& \leq C_T^2L_\psi^2\int_t^T
(r-t)^{-1/2}dr\left(\int_r^T(s-r)^{-1/2}\norma{G\nabla(u-v)(s,\cdot)}_\infty
ds\right) \\
& = C_T^2L_\psi^2\int_t^T\norma{G\nabla(u-v)(s,\cdot)}_\infty
ds\left(\int_t^s(r-t)^{-1/2}(s-r)^{-1/2}dr\right) \\
& = C_T^2L_\psi^2\pi\int_t^T\norma{G\nabla(u-v)(s,\cdot)}_\infty
ds \\
\end{split}
\]

Hence, by the Gronwall Lemma we deduce that $\norma*{G\nabla
(u-v)(t,\cdot)}_\infty=0$, for any $t\in[0,T)$. To conclude, it is enough to
observe that
\[
\begin{split}
\norma{u-v}_\infty
& \leq \norma*{\int_t^TS(r-t)\left(F(r,u)-F(r,v)\right)dr}_\infty \\
& \leq L_\psi\int_t^T\norma*{G\nabla (u(r,\cdot)-v(r,\cdot))}_\infty dr \\
& = 0.
\end{split}
\]
\end{proof}

Now, we prove the existence of a mild solution of problem
\eqref{eq:eq_Cauchy_NL}.

\begin{thm}
There exist $\delta<T$ such that the operator $\Gamma$, defined by
\begin{equation}
\Gamma(v)(t,x)=S(T-t)\varphi(x)-\int_t^TS(r-t)F(r,v)(x)dr, \quad
(t,x)\in (T-\delta,T]\times\R^N,
\label{eq:Gamma_1}
\end{equation}
for any $v\in \Kk_{\delta}$, admits a unique fixed point.
\label{thm:fixed_point_NL}
\end{thm}

\begin{proof}[Proof]
Set
\[
\Kk_{\delta,R}=\left\{
\begin{array}{l}
h\in C_b\left([T-\delta,T]\times \R^N\right)\cap
C^{0,1}\left([T-\delta,T)\times\R^N\right): \\
\\
\displaystyle\norma*{h}_{\Kk_\delta}\leq R
\end{array}
\right\},
\]
endowed with the norm $\norma\cdot_{\Kk_\delta}$ (see \eqref{eq:normakalpha}).
Since $\Kk_{\delta,R}\subset \Kk_\delta$, Proposition \ref{prop:uniqueness_1} shows
that if we find that $\Gamma$ is a contraction in $\Kk_{\delta,R}$ then its unique fixed point
is the unique mild solution to problem
\eqref{eq:eq_Cauchy_NL} which belongs to $\Kk_\delta$.

Hence we prove that $\Gamma(v)\in \Kk_{\delta,R}$ for any
$v\in \Kk_{\delta,R}$, and there exists $c<1$ such that
\[
\norma {\Gamma(u)-\Gamma(v)}_{\Kk_{\delta,R}}\leq c\norma{u-v}_{\Kk_{\delta,R}},
\quad \forall u,v\in \Kk_{\delta,R}.
\]

For this purpose, we set
\[
C_T:=\sup_{t\in(0,T]}t^{1/2}\norma{G\nabla S(t)}
\]
and recall that $\sup_{t\in[0,T]}\norma{S(t)}\leq 1$ since $\{S(t)\}_{t\geq0}$ is a contraction semigroup.

Then by the second inequality in \eqref{eq:F_lipschitz} we have
\begin{equation}
\begin{split}
\norma {\Gamma (v(t,\cdot))}_\infty
& \leq \norma \varphi_\infty +\norma*{\int_t^T
S(r-t)F(r,v)dr}_\infty \\
& \quad +\norma*{\int_t^TS(r-t)F(r,0)dr}_\infty \\
& \leq \norma \varphi_\infty+2L_\psi(T-t)^{1/2}\norma
v_{\Kk_{\delta,R}}+(T-t)L_\psi \\
& \leq \norma \varphi_\infty+2L_\psi\delta^{1/2}\norma
v_{\Kk_{\delta,R}}+\delta L_\psi.
\end{split}
\label{eq:fix_point1_intoitself1}
\end{equation}
and
\begin{equation}
\begin{split}
(T-t)^{1/2} & \norma*{G\nabla\Gamma (v(t,\cdot))}_\infty \\
& \leq C_T\norma
\varphi_\infty+(T-t)^{1/2}C_TL_\psi\int_t^T(r-t)^{-1/2}\left(\norma*{G\nabla
v(r,\cdot)}_\infty+1\right) dr \\
& \leq C_T\norma
\varphi_\infty+C_TL_\psi(T-t)^{1/2}\norma*{v}_{\Kk_{\delta,R}}\int_t^T(r-t)^{-1/2}
(T-r)^{-1/2}dr \\
& \quad +2(T-t)C_TL_\psi\\
& \leq C_T\norma
\varphi_\infty+\pi C_TL_\psi(T-t)^{1/2}\norma*{v}_{\Kk_{\delta,R}}+2(T-t)
C_TL_\psi
\\
& \leq C_T\norma
\varphi_\infty+\pi\delta^{1/2}C_TL_\psi\norma*{v}_{\Kk_{\delta,R}}+2\delta
C_TL_\psi.
\end{split}
\label{eq:fix_point1_intoitself2}
\end{equation}

Moreover,
\begin{equation}
\begin{split}
\norma*{\Gamma (u(t,\cdot))-\Gamma (v(t,\cdot))}_\infty
& \leq \int_t^T\norma*{S(r-t)\left(F(r,u)-F(r,v)\right)}_\infty dr \\
& \leq 2L_\psi\delta^{1/2}\norma{u-v}_{\Kk_{\delta,R}}
\end{split}
\label{eq:fix_point1_contraction1}
\end{equation}
and
\begin{equation}
\begin{split}
& (T-t)^{1/2}\norma{G\nabla \Gamma (u(t,\cdot))-G\nabla\Gamma
(v(t,\cdot))}_\infty \\
& \quad \leq (T-t)^{1/2}C_TL_\psi\int_t^T(r-t)^{-1/2}\norma*{G\nabla
u(r,\cdot)-G\nabla v(r,\cdot)}_\infty dr \\
& \quad \leq
(T-t)^{1/2}C_TL_\psi\norma{u-v}_{\Kk_{\delta,R}}\int_t^T(r-t)^{-1/2}(T-r)^{-1/2}dr
\\
& \quad \leq \pi(T-t)^{1/2}C_TL_\psi\norma{u-v}_{\Kk_{\delta,R}} \\
& \quad \leq \pi\delta^{1/2}C_TL_\psi\norma{u-v}_{\Kk_{\delta,R}}.
\end{split}
\label{eq:fix_point1_contraction2}
\end{equation}

Now we have to choose $\delta$ and $R$. Set
\[
\delta=(4L_\psi+2\pi C_TL_\psi)^{-2}\wedge T
\]
in \eqref{eq:fix_point1_contraction1} and \eqref{eq:fix_point1_contraction2}; it
immediately follows that
\[
\begin{split}
\norma{\Gamma(u)-\Gamma(v)}_{\Kk_{\delta,R}}
& \leq
2L_\psi\delta^{1/2}\norma{u-v}_{\Kk_{\delta,R}}+\delta^{1/2}
\pi C_TL_\psi\norma { u-v
}_{\Kk_{\delta,R}} \\
& = \delta^{1/2}\left(2L_\psi+\pi C_TL_\psi\right)\norma{u-v}_{\Kk_{\delta,R}}
\\
& \leq \frac{1}{2}\norma{u-v}_{\Kk_{\delta,R}},
\end{split}
\]
and so $\Gamma$ is a $1/2$-contraction. To show that $\Gamma$ maps
$X_{\delta,R}$ into itself, it is sufficient to take
\[
R=2\left(1+2C_T\right)\left(\norma
\varphi_\infty+\delta L_\psi\right).
\]

Indeed, substituting in \eqref{eq:fix_point1_intoitself1} and
\eqref{eq:fix_point1_intoitself2}, we get
\[
\begin{split}
\norma*{\Gamma(v)}_{\Kk_{\delta,R}}
& \leq \norma \varphi_\infty+2L_\psi\delta^{1/2}\norma
u_{\Kk_{\delta,R}}+\delta L_\psi \\
& \quad +C_T\norma
\varphi_\infty+\delta^{1/2}\pi C_TL_\psi\norma*{v}_{\Kk_{\delta,R}}+2\delta
C_TL_\psi
\\
& \leq \left(1+2C_T\right)\left(\norma
\varphi_\infty+\delta L_\psi\right)
+\delta^{1/2}\left(2L_\psi+\pi C_TL_\psi\right)\norma v_{\Kk_{\delta,R}} \\
& \leq \frac R2+\frac R2\leq R.
\end{split}
\]
\end{proof}

\begin{rmk}
If $\varphi\in C^1_b(\R^N)$ the same arguments as in the proof of Theorem
\ref{thm:fixed_point_NL} and Proposition \ref{prop:uniqueness_1} show that the
operator $\Gamma$ in \eqref{eq:Gamma_1} admits a unique fixed point in the
space $\Kk_\delta$ defined by
\[
\Kk_\delta=
\left\{
\begin{array}{l}
\displaystyle h\in C_b\left([T-\delta,T]\times\R^N\right)\cap
C^{0,1}\left([T-\delta,T]\times\R^N\right): \\
\\
\displaystyle \sup_{(t,x)\in(T-\delta,T)\times\R^N}
\abs* { G
(x)\nabla h(t,x)}<\infty.
\end{array}
\right\}
\]
for some $\delta>0$.
\label{rmk:solution_NL_1}
\end{rmk}

Now, we can construct the maximally defined solution of \eqref{eq:eq_Cauchy_NL}.
Set
\[
\left\{
\begin{array}{l}
\tau(\varphi)=\inf\{0<a<T:\textrm{ problem \eqref{eq:eq_Cauchy_NL} has a mild solution
$v_a$ in $\Kk_a$}\}, \\
\\
v(t,x)=v_a(t,x), \quad \textrm{if $t\geq T-a$}.
\end{array}
\right.
\]

The function $v$ is well defined, thanks to Theorem \ref{thm:fixed_point_NL},
in the interval
\[
I(\varphi)=\cup \{[T-a,T]:\textrm{ problem \eqref{eq:eq_Cauchy_NL} has a mild
solution
$v_a$ in $\Kk_a$}\},
\]
and we have $\tau(\varphi)=\inf I(\varphi)$.

\begin{prop}
If $\varphi\in C_b(\R^N)$ is such that $I(\varphi)\neq [0,T]$, and $F$ satisfies
\eqref{eq:F_lipschitz}, then the function
\[
t\mapsto (T-t)^{1/2}\norma{G\nabla v(t,\cdot)}_\infty
\]
is unbounded in $I(\varphi)$.
\label{prop:unboudedness_1}
\end{prop}

\begin{proof}[Proof]
Even if proof is rather classical, for the reader's convenience we provide the details. Let us suppose that the function
\[
t\mapsto (T-t)^{1/2}\norma{G\nabla v(t,\cdot)}_\infty
\]
is bounded in $I(\varphi)$, and let $v$ be the maximally defined solution to
\eqref{eq:eq_Cauchy_NL}. Moreover, we set $\tau(\varphi)=\tau$.
$S(\cdot)\varphi$ is continuous
in $(0,\infty)\times\R^N$, and by Proposition \ref{prop:regularity_mild_solution} the
function
\[
(t,x)\mapsto\int_t^T S(r-t)F(r,v)(x)dr
\]
is continuous and bounded in $[\tau,T]\times\R^N$. Hence, we can extend $v$ up
to $t=\tau$, defining
\[
v(\tau,x):=T(\tau)\varphi(x)-\int_\tau^T S(r-\tau)F(r,v)(x)dr.
\]

Since $v(\tau,\cdot)\in C_b(\R^N)$, by Theorem \ref{thm:fixed_point_NL} the
Cauchy problem
\[
\left\{
\begin{array}{ll}
w(t,x)+Aw(t,x)=\psi(x,G\nabla w(t,\cdot))(x), \quad t<\tau, & x\in\R^N, \\
\\
w(\tau,x)=v(\tau, x), & x\in\R^N,
\end{array}
\right.
\]
admits a unique mild solution in $[\tau-\delta,\tau]$, for some $\delta>0$. If
we define
\[
z(t,x)=
\begin{cases}
w(t,x), & \tau-\delta\leq t\leq\tau, \quad x\in\R^N, \\
\\
v(t,x), & \tau\leq t\leq T, \quad x\in\R^N, \\
\end{cases}
\]
then $z$ is a mild solution of \eqref{eq:eq_Cauchy_NL} in
$[\tau-\delta,T]\times\R^N$ which extends $v$, and it contradicts the
maximality of $v$.
\end{proof}

\begin{prop}
If $F$ satisfies \eqref{hyp:g_lipschitz}, then the mild solution $v$ of
problem \eqref{eq:eq_Cauchy_NL} exists in $[0,T]\times\R^N$.
\label{prop:linear_growth1}
\end{prop}

\begin{proof}[Proof]
By Proposition \ref{prop:unboudedness_1}, it is enough to show that the function
\[
(t,x)\mapsto (T-t)^{1/2}G(x)\nabla v(t,x)
\]
is bounded in $I(\varphi) \times \R^N$.

For sake of simplicity, we set
\[
l(t):=\norma{G\nabla v(t,\cdot)}_\infty,
\]
where $v$ is the maximally defined solution of problem \eqref{eq:eq_Cauchy_NL}.
Then for any $t\in I(\varphi)$ and $x\in\R^N$,
\[
\begin{split}
& (T-t)^{1/2}l(t) \\
& \quad \leq C_T\norma \varphi_\infty
+L_\psi\int_t^T(T-t)^{1/2}(r-t)^{-1/2}\left(1+l(r)\right)dr \\
& \quad \leq C_T\norma
\varphi_\infty+2TL_\psi \\
& \qquad +L_\psi(T-t)^{1/2}\int_t^T(r-t)^{-1/2}(T-r)^{-1/2}(T-r)^{1/2}l(r) dr \\
& \leq C_T\norma
\varphi_\infty+2TL_\psi \\
& \qquad +L_\psi(T-t)^{1/2}\int_t^T(r-t)^{-1/2}(T-r)^{-1/2}\left(C_T\norma
\varphi_\infty+2TL_\psi\right)dr \\
& \qquad +L_\psi^2(T-t)^{1/2}\int_t^T(r-t)^{-1/2} \left(\int_r^T(s-r)^{-1/2}(T-s)^{-1/2}(T-s)^{1/2}l(s)ds\right)dr \\
& \quad \leq (C_T\norma\varphi_\infty+2TL_\psi)(1+T^{1/2}\pi L_\psi) \\
& \qquad +\pi L_\psi^2(T-t)^{1/2}\int_t^T(T-s)^{-1/2}(T-s)^{1/2}l(s)ds.
\end{split}
\]

The generalized Gronwall Lemma guarantees that the function
$(t,x)\mapsto (T-t)^{1/2}G(x)\nabla v(t,x)$ is bounded in
$I(\varphi)\times \R^N$, and the thesis follows.
\end{proof}

\begin{rmk}
Since the problem \eqref{eq:eq_Cauchy_NL} is autonomous, in Propositions \ref{prop:unboudedness_1} and \ref{prop:linear_growth1} we can replace $[0,T]$ with $(-\infty,T]$.
\end{rmk}

\begin{rmk}
Under the Hypotheses of Proposition \ref{prop:linear_growth1}, if $\varphi\in C^1_b(\R^N)$ then the mild solution $v$ of problem \eqref{eq:eq_Cauchy_NL} exists in $(-\infty,T]\times\R^N$, it belongs to $C^{0,1}((-\infty,T]\times\R^N)$ and it is bounded in $(a,T]\times\R^N$, for any $a<T$.
\label{rmk:est_solution_NL_2}
\end{rmk}

%
%
%

\section{The Forward Backward Stochastic Differential Equation Associated to the
Semi-Linear PDE}
\label{sec:FBSDE}

Let $(\Omega,\F,\Pp)$ be a complete probability space, $(W_t)_{t\geq0}$ a
real Brownian motion and $\mathcal N$ the family of elements of $\F$ of
probability $0$. We define as $\F^W_t$ the natural filtration with respect to
$W_t$, completed by the $\Pp-$null set of $\F$, i.e.
\[
\F^W_t:=\sigma\{W_s:0\leq s\leq t,\ \mathcal N\}.
\]

In this setting we study the Forward Backward Stochastic Differential
Equation
\begin{equation}
\left\{
\begin{array}{ll}
dY_\tau=\psi(X_\tau,Z_\tau)d\tau+Z_\tau dW_\tau, & \tau\in[t,T], \\
\\
dX_\tau=B(X_\tau) d\tau+G(X_\tau) dW_\tau, & \tau\in[t,T], \\
\\
Y_T=\varphi(X_T), \\
\\
X_t=x, & x\in\R^N,
\end{array}
\right.
\label{eq:FBSDE}
\tag{FBSDE}
\end{equation}
where
\[
\psi:\R^N\times\R^N\longrightarrow\R, \quad \varphi:\R^N\longrightarrow\R,
\]
are given Borel functions, and
\[
B,G:\R^N\longrightarrow\R
\]
are Borel measurable.

For any $p\in[1,\infty)$, let $\Hh^p$ be the space of progressively measurable with respect to $\F_t^W$ random processes $X_t$
such
that
\[
\norma X_{\Hh^p}:= \E\sup_{t\in[0,T]}\abs{X_t}^p<\infty,
\]
and let $\K$ be the space of $(\F_t^W)-$progressively measurable processes $Y,Z$ such that
\[
\norma{(Y,Z)}^2_{{\it{cont}}}:=\E\sup_{t\in[0,T]}\abs{Y_t}^2+\E\int_0^T\abs{
Z_\sigma}^2d\sigma<\infty.
\]

Moreover, we denote by $Y(s,t,x)$ and $Z(s,t,x)$ the solution to \eqref{eq:FBSDE}.

Throughout this section we assume the following additional assumptions on $B$
and $G$:
\begin{hyp}
\label{hyp:growth_FBSDE}
There exists $C>0$ such that, for all $x,x',z,z'\in\R^N$, we have
\begin{equation}
\abs{B(x)-B(x')}+\abs{G(x)-G(x')}\leq C\abs{x-x'}.
\end{equation}
\end{hyp}

If Hypothesis \ref{hyp:growth_FBSDE} is satisfied and
\[
\abs{\varphi(x)}+\abs{\psi(x,0)} \leq C(1+\abs x), \quad \forall x\in\R^N.
\]
then system \eqref{eq:FBSDE} admits a unique solution $(X,Y,Z)$, where $X\in\Hh^p$,
for any $p\in[1,\infty)$, and $(Y,Z)\in\K$ (see \cite{pardoux-peng}). Henceforth, $X$ denotes the solution
to the forward equation in \eqref{eq:FBSDE}.

\begin{rmk}
The hypotheses on the growth of $B$ and $G$ in \ref{hyp:growth_FBSDE} are compatible with the growth conditions on the coefficients of the operator $A$ in Hypothesis \ref{hyp:maximum_principle} (see Example \ref{ex:coeffA1}).
\end{rmk}


The parabolic Cauchy problem studied in Section \ref{sec:BPDE}
\[
\left\{
\begin{array}{ll}
D_tv(t,x)+Av(t,x)=\psi(x,G(x)\nabla_xv(t,x)), & x\in\R^N, \quad t\in[0,T), \\
\\
v(T,x)=\varphi(x), & x\in\R^N,
\end{array}
\right.
\]
is strictly linked with \eqref{eq:FBSDE}. Indeed,
if $v\in C^{1,2}([0,T]\times\R^N)$ is a solution to \eqref{eq:eq_Cauchy_NL}, then
$v(t,x)=Y(t,t,x)$. Conversely, if $\psi,\varphi,B,G$, satisfy stronger
conditions, then, setting $v(t,x)=Y(t,t,x)$, it turns out that $v\in
C^{1,2}([0,T]\times\R^N)$ and it is a solution to \eqref{eq:eq_Cauchy_NL} (see \cite{pardoux-peng}).

We want to relax regularity conditions on $\psi$ and $\varphi$, and growth conditions on $B$ and $G$, and prove that $V$ is still a solution to \eqref{eq:FBSDE}. For this purpose, we will use the results in Section \ref{sec:BPDE}. Notice that since $G$ may be unbounded a straightforward application of Bismut-Elworthy formula as in \cite{cerrai} is not allowed.


Assume that $G, B, \psi$ satisfy Hypotheses \ref{hyp:g_lipschitz} and \ref{hyp:maximum_principle}. Moreover, suppose that $\varphi\in BUC(\R^N)$. Hence, by Theorem \ref{thm:fixed_point_NL} and
Proposition \ref{prop:linear_growth1}, there exists a unique solution $v$ to
\eqref{eq:eq_Cauchy_NL} in $[0,T]$ which belongs to $\Kk_T$ (see Definition
\ref{def:K_spaces}).

To use the result of \cite{pardoux-peng}, we approximate the functions $\varphi,\psi$ by
convolution: let $(\rho_n)_{n\in\N}$ be a standard sequence of mollifiers
in $\R^\N$ and set
\[
\varphi_n=\varphi\star\rho_n, \quad \psi_n=\psi\star_z\rho_n,
\]
where $\star_z$ denotes the convolution with respect only to the variable $z$.

$\psi_n$ and $\varphi_n$, are smooth functions and $\varphi_n$ are bounded. In particular, for any $n\in\N$ we have that $\norma {\varphi_n}_\infty\leq\norma\varphi_\infty$ and by \eqref{eq:g_lipschitz} we deduce that for any $n,m\in\N$ and $x,z_1,z_2\in\R^N$, it holds that
\begin{align}
\abs{\psi_n(x,z_1)-\psi(x,z_2)} & \leq L_\psi\abs{z_1-z_2}+\frac{L_\psi}n, \label{eq:prop_psi_n}
\\
\abs{\psi_n(x,z_1)-\psi_m(x,z_2)} & \leq
L_\psi\abs{z_n-z_m}+L_\psi\left(\frac1n+\frac1m\right). \label{eq:prop_psi_nm}
\end{align}

For any $n\in\N$, let us consider the approximated Cauchy problem
\begin{equation}
\left\{
\begin{array}{ll}
D_tv_n(t,x)+Av_n(t,x)=\psi_n(x,G(x)\nabla
v_n(t,x)), \quad t\in[0,T), & x\in\R^N, \\
\\
v_n(T,x)=\varphi_n(x), & x\in\R^N,
\end{array}
\right.
\label{eq:appr_Cauchy_problem1}
\end{equation}
whose mild solution is given by (see Theorem \ref{thm:fixed_point_NL})
\begin{equation}
\begin{split}
v_n(t,x)
& = S(T-t)\varphi_n(x)-\int_t^TS(r-t)\psi_n(x,G(x)\nabla v_n(r,x))dr \\
& = S(T-t)\varphi_n(x)-\int_t^TS(r-t)F_n(r,v_n)(x)dr, \\
\end{split}
\label{eq:mild_solution_n}
\end{equation}
where
\[
F_n:(0,T)\times \Kk_T\longrightarrow C_b(\R^N), \quad F_n(t,u)(x):=\psi_n(x,G(x)\nabla u(t,x)).
\]

Remarks \ref{rmk:solution_NL_1} and \ref{rmk:est_solution_NL_2} guarantee that $v_n\in C_b([0,T]\times\R^N)$ and
$\norma{G\nabla v_n(t,\cdot)}_\infty\leq C_n$, for any $t\in(0,T)$ and any $n\in\N$.

We recall that, since $\varphi\in C_b(\R^N)$ and the coefficients of $B,Q$ belong to
$C^\delta_{\it{loc}}(\R^N)$, $S(\cdot)f(\cdot)\in
C_{\it{loc}}^{1+\delta/2,2+\delta}((0,\infty)\times\R^N)$ (see Hypothesis \ref{hyp:maximum_principle}). Hence the Hypotheses in \cite{pardoux-peng} are
satisfied. It means that the function $v_n\in C^{0,1}([0,T]\times\R^N)$, and
$v_n(t,x)=Y^n(t,t,x)$, where $Y^n$ is the solution to
\begin{equation}
\left\{
\begin{array}{ll}
dY^n_\tau=\psi_n(X_\tau,Z^n_\tau)d\tau+Z^n_\tau dW_\tau, & \tau\in[t,T], \\
\\
dX_\tau=B(X_\tau) d\tau+G(X_\tau) dW_\tau, & \tau\in[t,T], \\
\\
Y^n_T=\varphi_n(X_T), \\
\\
X_t=x, & x\in\R^N.
\end{array}
\right.
\label{eq:FBSDE_n}
\end{equation}

Now we need to study how $v_n$ and $G\nabla v_n$ converge to $v$ and $G\nabla
v$, respectively.
We claim that, for any fixed $t\in[0,T)$, $v_n(t,\cdot)$ and $G\nabla
v_n(t,\cdot)$ converge uniformly. Then, we can define
\begin{equation}
Y(s,t,x):=v(s,X(s,t,x)), \quad Z(s,t,x):=G(X(s,t,x))\nabla v(s,X(s,t,x)),
\label{eq:identification_formulae}
\end{equation}
for any $t\in[0,T]$, $t\leq s< T$, and $x\in\R^N$. Finally, we will show
that $(X,Y,Z)$ is a solution to \eqref{eq:FBSDE}.

To prove the above claim, we need an intermediate result, contained in the
following lemma.

\begin{lem}
$[v_n]_{\Kk_T}$ is uniformly bounded.
\label{lem:unif_boundedness_vn}
\end{lem}

\begin{proof}[Proof]
Let $t\in[0,T)$. Since $\abs{\psi_n(x,0)}\leq L_\psi$, the same computations of Proposition \ref{prop:uniqueness_1} yield to the thesis.
\end{proof}

\begin{thm}
Suppose that Hypotheses \ref{hyp:g_lipschitz}, \ref{hyp:maximum_principle} and
\ref{hyp:growth_FBSDE} hold. Moreover, let $\varphi\in BUC(\R^N)$. Then, for any
$t\in[0,T)$, $v_n(t,\cdot)$ and $G\nabla v_n(t,\cdot)$ converge
uniformly to $v(t,\cdot)$ and $G(\cdot)\nabla v(t,\cdot)$
respectively.
Moreover, $(X,Y,Z)$ is a solution to \eqref{eq:FBSDE}, where $Y$ and $Z$ are
defined by \eqref{eq:identification_formulae}.
\label{thm:unif_convergence_vn_gradvn}
\end{thm}

\begin{proof}[Proof]
As usual, first we prove the convergence of $G\nabla
v_n$, since it is involved in the definition of $v_n$. To simplify the notations, we set
\[
h_n(t):=\norma{G\nabla v_n(t,\cdot)-G\nabla
v(t,\cdot)}_\infty.
\]

We have
\[
\begin{split}
& (T-t)^{1/2}h_n(t) \\
& \quad \leq
C_T\norma{\varphi_n-\varphi}_\infty+C_T(T-t)^{1/2}\int_t^T(r-t)^{-1/2}\norma{F_n(r,
v_n)-F(r,v)}_\infty dr \\
& \quad \leq
C_T\norma{\varphi_n-\varphi}_\infty+C_T(T-t)^{1/2}\int_t^T(r-t)^{-1/2}\norma{F_n(r,
v_n)-F(r,v_n)}_\infty dr \\
& \qquad +C_T(T-t)^{1/2}\int_t^T(r-t)^{-1/2}\norma{F(r,v_n)-F(r,v)}_\infty ds \\
& \quad =: I^n_1+I^n_2(t) \\
& \qquad +C_TL_\psi(T-t)^{1/2}\int_t^T(r-t)^{-1/2}(T-r)^{-1/2}(T-r)^{1/2}h_n(t) dr.
\end{split}
\]

Now we use the estimate
\[
I^n_2(t)\leq C_T\frac{L_\psi}n\int_t^T(r-t)^{-1/2}dr=2C_T\frac{L_\psi}nT^{1/2},
\]
which follows from \eqref{eq:prop_psi_n} with $z_1=z_2$ and holds for any $t\in[0,T)$. Hence
\[
\begin{split}
& \quad \leq I^n_1+2C_T\frac{L_\psi}nT^{1/2} \\
& \qquad +C_TL_\psi(T-t)^{1/2}\int_t^T(r-t)^{-1/2}(T-r)^{-1/2}\left(I^n_1+2C_T\frac{L_\psi}nT^{1/2}\right)dr \\
& \qquad +C_T^2L_\psi^2(T-t)^{1/2}\int_t^T(r-t)^{-1/2}\left(\int_r^T(r-s)^{-1/2}(T-s)^{-1/2}((T-s)^{1/2}h_n(r) ds\right)dr \\
& \quad \leq \left(I^n_1+2C_T\frac{L_\psi}nT^{1/2}\right)(1+\pi C_TL_\psi T^{1/2}) \\
& \qquad +\pi C_T^2L_\psi^2(T-t)^{1/2}\int_t^T(T-s)^{-1/2}(T-s)^{1/2}h_n(r) ds.
\end{split}
\]

Since $\varphi\in BUC(\R^N)$, $I^n_1$ tends to zero, as $n\rightarrow+\infty$. Clearly, also
\[
2C_T\frac{L_\psi}nT^{1/2}
\]
vanishes as $n\rightarrow \infty$

Now we apply the generalized Gronwall Lemma to the function
\[
(T-t)^{1/2}\norma{G\nabla v_n(t,\cdot)-G\nabla
v(t,\cdot)}_\infty,
\]
and we obtain
\[
(T-t)^{1/2}\norma{G\nabla v_n(t,\cdot)-G\nabla v(t,\cdot)}_\infty
\leq\left(I^n_1+2C_T\frac{L_\psi}nT^{1/2}\right)\exp{\left(\pi C_T^2L_\psi^2T\right)},
\]
and the right-hand side tends to zero, as $n\rightarrow+\infty$.

Using the fact that $[v_n-v]_{\Kk_T}$ tends to zero, similar computations yield
the uniformly convergence of $v_n(t,\cdot)$ to $v(t,\cdot)$, for any
$t\in[0,T]$.

Finally, we prove that the processes $Y,Z$ defined in
\eqref{eq:identification_formulae} are solutions to \eqref{eq:FBSDE}.
Since $Y_n,Z_n$ are solutions of \eqref{eq:FBSDE_n}, and the equalities hold
$\Pp-$a.s., there exists a family of elements of $\F$, $\{\Omega_n\}$,
such that each of them has zero measure. Moreover, if we set
$\tilde\Omega=\cup_n\Omega_n$, then $\Pp(\tilde\Omega)=0$, and in
$\tilde\Omega^c$ \eqref{eq:FBSDE_n} holds, for any
$n\in\N$.

Now we fix $x\in\R^N$, $t\in[0,T]$, set $X_\tau:=X(\tau,t,x)$, and define
\[
Y_\tau=v(\tau,X_\tau), \
Y^n_\tau=v^n(\tau,X_\tau), \  Z_\tau=G(X_\tau)\nabla
v(\tau,X_\tau), \
Z^n_\tau=G(X_\tau)\nabla v_n(\tau,X_\tau),
\]
for any $\tau\in[t,T]$. The previous estimates guarantee that
\[
Y^n_\tau \longrightarrow Y_\tau, \quad \varphi_n(X_T) \longrightarrow\varphi(X_T),
\]
uniformly in $\Omega$, and
\[
\int_\tau^T\psi_n(X_\sigma,Z^n_\sigma)d\sigma \longrightarrow
\int_\tau^T\psi(X_\sigma,Z_\sigma)d\sigma.
\]

Indeed, by
\eqref{eq:prop_psi_nm} we deduce that
\[
\begin{split}
\abs{\psi_n(X_\sigma,Z^n_\sigma)-\psi(X_\sigma,Z_\sigma)} & \leq
L_\psi\abs{Z^n_\sigma-Z_\sigma}+\frac{L_\psi}n, \\
\abs{\psi(X_\sigma,Z_\sigma)},\ \abs{\psi_n(X_\sigma,Z^n_\sigma)} &\leq
L_\psi C(1+(T-\sigma)^{-1/2}),
\end{split}
\]
for any $x\in\R^N$ and $\sigma\in[\tau,T)$. Since $\abs*{Z^n_\sigma-Z_\sigma}$
tends to zero uniformly in $\Omega$, as $n\rightarrow+\infty$, and $\abs{\psi(X_\sigma,Z_\sigma)},\
\abs{\psi_n(X_\sigma,Z^n_\sigma)}$ can be estimated by an integrable function, we
can apply dominated convergence to the integral term.

It remains to prove the convergence of $\int_\tau^T Z^n_\sigma dW_\sigma$ to
$\int_\tau^T Z_\sigma dW_\sigma$. At first, we prove that $\int_\tau^T
Z_\sigma dW_\sigma$ makes sense, since this is not guaranteed by previous
estimates, which show only that the growth $Z_\sigma$ can be estimated by
$(T-\sigma)^{-1/2}$, which is not square integrable in $T$.

We are going to show that $\{Z^n_\tau\}$ is a Cauchy sequence in the space
$L^2(\Omega\times(0,T))$, the space of the square integrable processes $V$,
endowed with the norm $\E\int_0^T\abs{V_\sigma}^2d\sigma$. Since this is a
Hilbert space, $\{Z^n_\tau\}$ converges to a process $\tilde Z_\tau$ which is
square integrable, and so, up to a subsequence, $\{Z^n_\tau\}$ converges to
$\tilde Z_\tau$ $[0,T]\otimes\Pp-$a.s. But $\{Z^n_\tau\}$ converges to $Z_\tau$
uniformly, hence pointwise, for any $\tau\in[0,T]$. Therefore, $\tilde
Z_\tau=Z_\tau$ $\Pp-$a.s., for almost every $\tau\in[0,T]$. This means
that $Z_\sigma$ is a square integrable process.

For the reader's convenience, we introduce some new notations:
\[
\begin{split}
\overline Y^{n,m}_\sigma & :=Y^n_\sigma-Y^m_\sigma, \\
\overline Z^{n,m}_\sigma & :=Z^n_\sigma-Z^m_\sigma, \\
\overline\varphi^{n,m}_\sigma & := \varphi_n(X_\sigma)-\varphi_m(X_\sigma), \\
\overline\psi^{n,m}_\sigma & :=
\psi_n(X_\sigma,Z^n_\sigma)-\psi_m(X_\sigma,Z^m_\sigma),
\end{split}
\]
for any $n,m\in\N$, $\sigma\in[0,T]$. By the It\^o formula, we get
\[
\begin{split}
d\abs{\overline Y^{n,m}_\tau}^2=-2\overline
Y^{n,m}_\tau\overline \psi^{n,m}_\tau d\tau
-2\overline Y^{n,m}_\tau\overline Z^{n,m}_\tau dW_\tau+\abs{\overline
Z^{n,m}_\tau}^2d\tau,
\end{split}
\]
and, recalling that $\overline Y^{n,m}_T=\overline\varphi^{n,m}_T$, we obtain
\[
\abs{\overline Y^{n,m}_\tau}^2+\int_\tau^T\abs{\overline Z^{n,m}_\sigma}^2
d\sigma=\abs{\overline\varphi^{n,m}_T}^2
-2\int_\tau^T\overline Y^{n,m}_\sigma\overline\psi^{n,m}_\sigma d\sigma
-2\int_\tau^T \overline Y^{n,m}_\sigma \overline Z^{n,m}_\sigma dW_\sigma.
\]

Let us estimate the terms in the right-hand side. Note that
$(Y^n,Z^n),(Y^m,Z^m)\in\K$, since they are solutions of a backward stochastic
differential equation. Hence, the process $I_\tau=\int_0^\tau\overline
Y^{n,m}_\sigma \overline Z^{n,m}_\sigma dW_\sigma$ is a martingale. Indeed
\[
\begin{split}
\E\left(\int_0^\tau\abs{\overline Y^{n,m}_\sigma \overline
Z^{n,m}_\sigma}^2d\sigma\right)^{1/2}
& \leq c\E\left(\sup_{\tau\in[0,T]}\abs{\overline
Y^{n,m}_\tau}^2+\int_0^\tau\abs{\overline Z^{n,m}_\sigma}^2
d\sigma\right)<+\infty.
\end{split}
\]

In particular $\E I_\tau=0$, for any $\tau$. Computing the expectation, we get
\begin{equation}
\E\abs{\overline Y^{n,m}_\tau}^2+\E\int_\tau^T\abs{\overline Z^{n,m}_\sigma}^2
d\sigma=\E\abs{\overline\varphi^{n,m}_T}^2
-2\E\int_\tau^T\overline Y^{n,m}_\sigma\overline\psi^{n,m}_\sigma d\sigma.
\label{eq:exp_nm}
\end{equation}

Moreover, by \eqref{eq:prop_psi_nm}, the last term in the right-hand side of \eqref{eq:exp_nm} can be estimated as
follows:
\[
\begin{split}
\E\int_\tau^T\abs{\overline Y^{n,m}_\sigma\overline\psi^{n,m}_\sigma} d\sigma
& \leq \E\left(\sup_{\tau\in[0,T]}\abs{\overline Y^{n,m}_\tau}
\int_\tau^T\abs{\overline \psi^{n,m}_\sigma} d\sigma \right)\\
& \leq 2L_\psi\sup_{n\in\N}\norma {v_n}_\infty
\left(\E\int_\tau^T\abs{\overline Z^{n,m}_\sigma}d\sigma+\frac Tm+\frac
Tn\right) \\
& \leq c\left(\E\int_\tau^T\abs{\overline Z^{n,m}_\sigma}d\sigma+\frac Tm+\frac
Tn\right).
\end{split}
\]

By the definitions of $Z^n,Z^m,\overline Z^{n,m}$ and the above
estimates, it is easy to prove, using dominated convergence, that, for any
$\epsilon>0$, there exists $\bar n\in\N$ such that $\E\int_0^T\abs{\overline
Z^{n,m}_\sigma}d\sigma\leq \epsilon$, for any $n,m\geq \bar n$.

The same arguments can be applied to $\overline\varphi^{n,m}_T$. Indeed, recalling
that $\varphi$ is uniformly continuous, for any $\epsilon>0$ there exists $\bar
n\in\N$ such that $\E\abs{\overline\varphi^{n,m}_T}^2\leq \epsilon$, for any
$n,m\geq \bar n$.

Hence $\{Z^n_\tau\}$ is a Cauchy sequence, and this implies that $\int_\tau^T
Z_\sigma dW_\sigma$ makes sense. Moreover, since $Z^n$ converges to $Z$ in
$L^2(\Omega\times(0,T))$, we see that
\[
\E\abs*{\int_\tau^T(Z^n_\sigma-Z_\sigma)
dW_\sigma}^2\longrightarrow 0, \quad n\rightarrow\infty.
\]

We can conclude that $\int_\tau^TZ^n_\sigma dW_\sigma$ tends to
$\int_\tau^TZ_\sigma dW_\sigma$ $\Pp-$a.s., and passing to the limit
\eqref{eq:FBSDE_n}, we obtain that the processes $(X,Y,Z)$ are a solution to
\eqref{eq:FBSDE} $\Pp-$a.s.
\end{proof}

\section{An application to the Stochastic Optimal Control in Weak Formulation}

In this section we consider the controlled equation \index{controlled equation}
\begin{equation}
\left\{
\begin{array}{ll}
d_\tau X_\tau=B(X_\tau)d\tau+G(X_\tau)r(X_\tau,u_\tau)d\tau+G(X_\tau)dW_\tau, &
\tau\in[t,T], \\
\\
X_t=x\in\R^N,
\end{array}
\right.
\label{eq:controlled_eq1}
\end{equation}
and the cost functional
\begin{equation}
\E\int_0^Tl(X_t,u_t)dt+\E\varphi(X_T),
\label{eq:cost_functional1}
\end{equation}
where $u$ is a progressive measurable stochastic process with values in some specified set
$\U\subset\R^N$, $r:\R^N\times \U\longrightarrow \R$, $W$ is a $\R^N-$valued
cylindrical Wiener process, and $l:\R^N\times\U\longrightarrow \R$.
Our purpose is to minimize over all admissible controls the cost functional.

We assume the following hypotheses on $l$ and $r$:
\begin{hyp}
There exists $C>0$ such that for all $x,x'\in\R^N,t\in[0,T],u,u'\in\U$, we have
\begin{equation}
\begin{split}
\abs{l(x,u)-l(x',u')}+\abs{r(x,u)-r(x',u')} & \leq
C\left(\abs{x-x'}+\abs{u-u'}\right), \\
\abs{l(x,u)}+\abs{r(x,u)} & \leq C.
\end{split}
\label{hyp:growth_ocp}
\end{equation}
\end{hyp}

\begin{defin}
An admissible control system (acs) $\Uu$ is the set
\[
\Uu=(\widehat\Omega,\widehat\F,(\widehat \F_t)_{t\geq0},\widehat\Pp,\widehat
u, \widehat W, \widehat X),
\]
where $(\widehat\Omega,\widehat\F,\Pp)$ is a probability space, the filtration
$(\widehat\F_t)_{t\geq0}$ verifies the usual conditions, the process $\widehat
W:[0,T]\times\widehat\Omega\longrightarrow\R^N$ is a Wiener process with respect
to $(\widehat\F_t)_{t\geq0}$, $\widehat u$ is progressive measurable with respect to
the filtration $(\widehat\F_t)_{t\geq0}$, and $\widehat X_\tau$ is a
solution to
\[
\widehat X_\tau=x+\int_t^\tau B(\widehat X_\sigma)d\sigma+\int_t^\tau
G(\widehat X_\sigma)r(\widehat X_\sigma,\widehat
u_\sigma)d\sigma+\int_t^\tau
G(\widehat X_\sigma)d\widehat W_\sigma, \quad \tau\in[t,T].
\]
\label{defin:acs}
\end{defin}

In this setting, the cost functional has the form
\begin{equation}
J(t,x,\Uu)=\widehat\E\int_t^Tl(\widehat X_\sigma,\widehat
u_\sigma)d\sigma+\widehat\E\varphi(\widehat X_T).
\label{eq:cost_functional_WF}
\end{equation}

An acs is called {\it{optimal}} for the control problem starting from $x$ at the
time $t$, if it minimizes $J(t,x,\cdot)$, and the minimum value of the cost is
called the {\it{optimal cost}}. Finally, we introduce the value function
$V:[0,T]\times\R^N\rightarrow\R$, defined by
\begin{equation}
V(t,x):=\inf_{u\in\Uu}J(t,x,u), \quad t\in[0,T], \quad x\in\R^N.
\label{eq:value_function}
\end{equation}

The Hamiltonian function of the problem, defined below, is crucial in the
analysis of the stochastic control problem.

\begin{defin}
The function $\psi:\R^N\times\R^N\longrightarrow\R$, defined by
\begin{equation}
\psi(x,z)=\inf_{u\in\U}\{l(x,u)+zr(x,u)\},
\label{eq:ham_function}
\end{equation}
is called {\it{Hamiltonian function}}.
\label{defin:ham_function}
\end{defin}

\begin{lem}
There exists a positive constant $c$ such that
\[
\begin{split}
\abs{\psi(x,0)} & \leq c, \\
\abs{\psi(x,z)-\psi(x',z')} & \leq c\abs{z-z'}+c\abs{x-x'}\left(1+\abs z+\abs{z'}\right),
\end{split}
\]
for any $x,x',z,z'\in\R^N$.
\label{lem:_prop_ham_function}
\end{lem}

\begin{proof}[Proof]
The result is well known, we report the proof for the reader's convenience.
We prove only the second inequality. For all $u\in\U$ we have
\[
\begin{split}
l(x,u)+zr(x,u)
& \leq l(x',u)+z'r(x',u)+\abs{l(x,u)-l(x',u)} \\
& \quad +\abs{zr(x,u)-z'r(x',u)} \\
& \leq l(x',u)+z'r(x',u)+\abs{l(x,u)-l(x',u)} \\
& \quad +\abs{zr(x,u)-z'r(x,u)}+\abs{z'r(x,u)-z'r(x',u)} \\
& \leq l(x',u)+z'r(x',u)+c\abs{x-x}+c\abs{z-z'}+c\abs{x-x'}\abs{z'}.
\end{split}
\]

Taking the infimum over $u$ and exchanging $x,z$ with $x',z'$ we get the conclusion.
\end{proof}

To prove the main theorem of this section, we need the following hypothesis:
\begin{hyp}
\label{hyp:min_acs}
For any $x,z\in\R^N$, the minimum in \eqref{eq:ham_function} is attained.
\end{hyp}

\begin{rmk}
\label{rmk:min_ham_function}
The minimum in \eqref{eq:ham_function} is always attained if $\U$ is a compact
set, see \cite{aubin-fran}.
\end{rmk}

\begin{rmk}
\label{rmk:ex_gamma_function}
If Hypothesis \ref{hyp:min_acs} is satisfied, then Filippov Theorem guarantees
that there exists a measurable function $\gamma:\R^N\times \R^N\rightarrow\U$
such that
\begin{equation}
\psi(x,z)=l(x,\gamma(x,z))+zr(x,\gamma(x,z)), \quad \forall x,z\in\R^N.
\label{eq:min_acs}
\end{equation}
\end{rmk}

Section \ref{sec:BPDE}
assures that the Hamilton Jacobi Bellman equation, associated to the problem
\eqref{eq:controlled_eq1} and \eqref{eq:cost_functional1}, admits a unique
solution $v$ in the space $\Kk_T$. We stress that this solution has a good
regularity, but not the optimal one; hence, we can not use the It\^o formula.
However, the BSDE's techniques enable us to prove that $v$ is indeed the value
function  of the problem, and has enough regularity to identify the optimal
feedback law.

\begin{thm}
Let Hypotheses \ref{hyp:g_lipschitz}, \ref{hyp:maximum_principle},
\ref{hyp:growth_FBSDE}, \ref{hyp:min_acs} and \ref{hyp:growth_ocp} hold. Moreover, let $\varphi\in BUC(\R^N)$. Then the following properties are satisfied:
\begin{description}

\item [(i)] there exists a unique solution $v$ of HJB such that $v\in \Kk_T$. Hence,
$G(x)\nabla v(t,x)$ is defined for any $t\in[0,T),x\in\R^N$;

\item [(ii)] $v(t,x)\leq V(t,x)$, for any $t\in[0,T],x\in\R^N$;

\item [(iii)] $v(t,x)=V(t,x)$ if and only if there exists an acs $\Uu^*$
such that
\begin{equation}
\psi(X^{\Uu^*}_t,Z_t)=l(X^{\Uu^*}_t,u^*_t)+Z_tr(X^{\Uu^*}_t,u^*_t),
\label{eq:cond_esist_min}
\end{equation}
where $X^{\Uu^*}_t$ is the solution to \eqref{eq:controlled_eq1}, with $u=u^*$;

\item [(iv)] there exists an acs $U^\#$ such that \eqref{eq:cond_esist_min} is
satisfied.
\end{description}
\end{thm}

\begin{proof}
For the reader's convenience we report the proof, which is closed to the one in \cite{fuhr-tess2}.

$(i)$: since the HJB equation associated to \eqref{eq:controlled_eq1} and
\eqref{eq:cost_functional1} is \eqref{eq:eq_Cauchy_NL}, the existence and
uniqueness of the mild solution follow from Section \ref{sec:BPDE}.

$(ii)$: we fix an acs $\Uu$, $t\in[0,T]$, $x\in\R^N$, and consider the
equation
\[
X^{\Uu}_\tau=x+\int_t^\tau B(X^{\Uu}_\sigma)d\sigma+\int_t^\tau
G(X^{\Uu}_\sigma)r(X^{\Uu}_\sigma,u_\sigma)d\sigma+\int_t^\tau
G(X^{\Uu}_\sigma)dW_\sigma, \quad \tau\in[t,T].
\]

Since $r$ is bounded, by Girsanov Theorem there exists a probability measure $\widetilde\Pp$ such that
\[
\widetilde W_\tau=W_\tau+\int_t^{t\wedge\tau}r(X^{\Uu}_\sigma,u_\sigma)d\sigma
\]
is a Wiener process with respect to $\widetilde\Pp$, and $X^{\Uu}$ is a solution
to
\[
X^{\Uu}_\tau=x+\int_t^\tau B(X^{\Uu}_\sigma)d\sigma+\int_t^\tau
G(X^{\Uu}_\sigma)d\widetilde W_\sigma, \quad \tau\in[t,T].
\]

Notice that $X^{\Uu}$ is measurable with respect to the $\sigma-$field
generated by $\widetilde W$.
Now we introduce the backward equation
\[
\widetilde Y_\tau+\int_t^\tau \widetilde Z_\sigma d\widetilde W_\sigma =
\varphi(X^{\Uu}_T)+\int_t^\tau\psi(X^{\Uu}_\sigma,\widetilde
Z_\sigma) d\sigma.
\]

By the Theorem \ref{thm:unif_convergence_vn_gradvn} there exists a unique solution
$(\widetilde Y,\widetilde Z)$ of this equation. Writing the backward equation
with respect to $W$, we get
\begin{equation}
\widetilde Y_\tau+\int_\tau^T \widetilde Z_\sigma dW_\sigma
+\int_\tau^T\widetilde Z_\sigma r(X^{\Uu}_\sigma,u_\sigma)d\sigma=
\varphi(X^{\Uu}_T)+\int_\tau^T\psi(X^{\Uu}_\sigma,\widetilde Z_\sigma)d\sigma.
\label{eq:back_eq_BM}
\end{equation}

By easy computations, we have that
$\E\left(\int_0^T\abs{\widetilde Z_t}^2dt\right)^{1/2}<\infty$. Hence, taking
the expectation in \eqref{eq:back_eq_BM} with respect to $\Pp$ and $\tau=t$, we
obtain
\[
\widetilde Y_t=\E\varphi(X^{\Uu}_T)+\E\int_t^T
\left[\psi(X^{\Uu}_\sigma,\widetilde Z_\sigma)-\widetilde
Zr(X^{\Uu}_\sigma,u_\sigma)\right]d\sigma.
\]

Adding and subtracting $\E\int_t^T l(X^{\Uu}_\sigma,u_\sigma)d\sigma$, and
recalling that $v(t,x)=\widetilde Y(t,t,x)$, we get
\begin{equation}
v(t,x)=J(t,x,\Uu)+\E\int_t^T\left[\psi(X^{\Uu}_\sigma,\widetilde
Z_\sigma)-\widetilde Z_\sigma
r(X^{\Uu}_\sigma,u_\sigma)-l(X^{\Uu}_\sigma,u_\sigma)\right]d\sigma.
\label{eq:HJB_costfunctional}
\end{equation}

From the definition of $\psi$, the term in square brackets is non positive.
Hence $v(t,x)\leq J(y,x,\Uu)$ for any acs $\Uu$, and taking the minimum we
deduce that
\[
v(t,x)\leq V(t,x), \quad  t\in[0,T], \ x\in\R^N.
\]

$(iii)$: from \eqref{eq:HJB_costfunctional}, it is clear that
$v(t,x)=J(t,x,\Uu^*)$ if and only if the acs
$\Uu^*$ satisfies \eqref{eq:cond_esist_min}. In this case, the integral term in
\eqref{eq:HJB_costfunctional} is zero; hence
\[
v(t,x)\leq V(t,x)\leq J(t,x,\Uu^*)=v(t,x).
\]


$(iv)$: by Hypothesis \ref{hyp:min_acs} and \eqref{eq:identification_formulae}, it is natural to define
\[
\widetilde \gamma(x)=\gamma (x,G(x)\nabla v(t,x)), \ t\in[0,T), \ x\in\R^N.
\]

Notice that $\widetilde \gamma$ is, a priori, not regular. Let $W$ be an
$N-$dimensional Brownian Motion on $(\Omega,\F,\{\F_t\}_t,\Pp)$, and $X^\#$ be
the solution to
\[
\left\{
\begin{array}{ll}
dX^\#_\tau=B(X^\#_\tau)d\tau+G(X^\#_\tau)dW_\tau, & \tau\in[t,T], \\
\\
X(t)=x\in\R^N.
\end{array}
\right.
\]

For any $\tau\in[t,T]$, we set
\[
W^\#_\tau=W_\tau-\int_t^{t\wedge\tau}r(X^\#_\sigma,
\widetilde\gamma(X^\#_\sigma))d\sigma;
\]
then $X^\#$ satisfies the closed-loop equation
\[
X^\#_\tau=x+\int_t^\tau B(X^\#_\sigma)d\sigma+\int_t^\tau
G(X^\#_\sigma)r(X^\#_\sigma,\widetilde\gamma(X^\#_\sigma))d\sigma+\int_t^\tau
G(X^\#_\sigma)dW^\#_\sigma,
\]
for any $\tau\in[t,T]$. Clearly, $\Uu^\#=(\Omega,\F,\{\F_t\}_t,\Pp,\widetilde\gamma(X^\#),X^\#,w^\#)$ is
an acs with $u^\#=\widetilde\gamma(X^\#)$. Moreover, $u^\#$ satisfies
\eqref{eq:cond_esist_min}: indeed
\[
\begin{split}
\psi(X^\#_\tau,Z^\#_\tau)
& = l(X^\#_\tau,\gamma(X^\#_\tau,Z^\#_\tau))+Z^\#_\tau r(X^\#_\tau,\gamma(X^\#_\tau,Z^\#_\tau)) \\
& = l(X^\#_\tau,\widetilde\gamma(X^\#_\tau))+Z^\#_\tau r(X^\#_\tau,\widetilde\gamma(X^\#_\tau)) \\
& = l(X^\#_\tau,u^\#_\tau)+Z^\#_\tau r(X^\#_\tau,u^\#_\tau),
\end{split}
\]
where $Z^\#_\tau=G(X^\#_\tau)\nabla v(\tau,X^\#_\tau)$. Hence $\Uu^\#$ is an
optimal control system for the problem.
\end{proof}

\end{document}